\documentclass{birkjour}

\usepackage{amssymb}
\usepackage{amsmath}
\usepackage[all]{xy}
\usepackage{graphicx}
\usepackage{multicol}
\usepackage{xcolor}
\usepackage{hyperref}
\usepackage{mathrsfs}

 \newtheorem{thm}{Theorem}[section]
 \newtheorem{cor}[thm]{Corollary}
 \newtheorem{lem}[thm]{Lemma}
 \newtheorem{prop}[thm]{Proposition}
 \newtheorem{fact}[thm]{Fact}

 \theoremstyle{definition}
 
\newtheorem{ex}[thm]{Example}
 \newtheorem{rem}[thm]{Remark}
   
  \newtheorem{ques}[thm]{Question}
 \numberwithin{equation}{section}

\newcommand{\Hom}{\mathrm{Hom}}

\newcommand{\Ext}{\mathrm{Ext}}
\newcommand{\Tor}{\mathrm{Tor}}

\newcommand{\image}{\mathrm{im}}
\newcommand{\kernel}{\mathrm{ker}}

\newcommand{\der}{\partial}
\newcommand{\eqo}{\mathbf{1}}
\newcommand{\iid}{\mathrm{id}}
\newcommand{\eps}{\varepsilon}
\newcommand{\rk}{\mathrm{rk}}

\newcommand{\cl}{\mathrm{cl}}

\newcommand{\ccd}{\mathrm{cd}}
\newcommand{\vcd}{\mathrm{vcd}}

\newcommand{\ifl}{\mathrm{inf}}
\newcommand{\rst}{\mathrm{res}}
\newcommand{\dfl}{\mathrm{def}}

\newcommand{\ob}{\mathrm{ob}}

\newcommand{\PSL}{\mathrm{PSL}}

\newcommand{\hotimes}{\widehat{\otimes}}
\newcommand{\baG}{\bar{G}}
\newcommand{\baO}{\bar{O}}
\newcommand{\bone}{\mathbf{1}}
\newcommand{\tG}{\tilde{G}}
\newcommand{\ttheta}{\tilde{\theta}}
\newcommand{\tH}{\tilde{H}}
\newcommand{\bSigma}{\bar{\Sigma}}
\newcommand{\ca}[1]{\mathcal{#1}}

\newcommand{\cts}{\mathrm{cts}}

\newcommand{\F}{\mathbb{F}}
\newcommand{\N}{\mathbb{N}}
\newcommand{\Z}{\mathbb{Z}}
\newcommand{\II}{\mathbb{I}}
\newcommand{\Q}{\mathbb{Q}}

\newcommand{\K}{\mathbb{K}}
\newcommand{\bK}{\bar{\K}}

\newcommand{\caO}{\mathcal{O}}

\newcommand{\ccl}{\mathrm{cl}}

\newcommand{\chr}{\mathrm{char}}

\newcommand{\Gal}{\mathrm{Gal}}
\newcommand{\ZpGK}{\Z_p\dbl G_{\K}\dbr}

\newcommand{\triv}{\{1\}}

\newcommand{\dbl}{[\![}
\newcommand{\dbr}{]\!]}
\newcommand{\ZpG}{\Z_p\dbl G\dbr}

\newcommand{\prf}{\mathbf{prf}}
\newcommand{\ZpGprf}{{}_{\ZpG}\prf}

\newcommand{\FQAlg}{{}_\F\mathbf{qalg}}

\newcommand{\bA}{\mathbf{A}}
\newcommand{\bB}{\mathbf{B}}
\newcommand{\boT}{\mathbf{T}}

\newcommand{\grp}{\mathrm{grp}}
\newcommand{\Aut}{\mathrm{Aut}}
\newcommand{\Zen}{\mathrm{Z}}

\newcommand{\baQ}{\bar{Q}}
\newcommand{\bgamma}{\bar{\gamma}}
\newcommand{\teta}{\tilde{\eta}}
\newcommand{\bh}{\bar{h}}
\newcommand{\tpi}{\tilde{\pi}}

\newcommand{\ppp}{\amalg^p}

\newcommand{\btheta}{\bar{\theta}}
\newcommand{\htheta}{\widehat{\theta}}

\newcommand{\euB}{\mathfrak{B}}
\newcommand{\argu}{\hbox to 7truept{\hrulefill}}

\begin{document}

%-------------------------------------------------------------------------


\title[Cyclotomic {$p$}-orientations]
{Profinite groups \\
		with a cyclotomic {$p$}-orientation}

%----------Author 1
\author[C. Quadrelli]{Claudio Quadrelli}

\address{%
Department of Mathematics and Applications\\
University of Milan-Bicocca\\
20125 Milan\\
Italy - EU}

\email{claudio.quadrelli@unimib.it}

%----------Author 2
\author{Thomas S. Weigel}
\address{%
Department of Mathematics and Applications\\
University of Milan-Bicocca\\
20125 Milan\\
Italy - EU}

\email{thomas.weigel@unimib.it}

\thanks{Both authors were partially supported by the PRIN 2015 ``Group Theory and Applications''. The first-named author was also partially supported by the Israel Science Fundation (grant No. 152/13)}

%%----------classification, keywords, date



\date{today}
\dedicatory{To the memory of Vladimir Voevodsky}
%%% ----------------------------------------------------------------------

\begin{abstract}
Let $p$ be a prime. 
A continuous representation $$\theta\colon G\to\mathrm{GL}_1(\Z_p)$$ of a profinite group $G$ is called a cyclotomic $p$-orientation
if for all open subgroups $U\subseteq G$ and for all $k,n\geq1$ the natural maps 
$$H^k(U,\Z_p(k)/p^n)\longrightarrow H^k(U,\Z_p(k)/p)$$ are surjective.
Here $\Z_p(k)$ denotes the $\Z_p$-module
of rank 1 with $U$-action induced by $\theta\vert_U^k$.
By the Rost-Voevodsky theorem, the cyclotomic character of the absolute Galois group $G_{\K}$ of a field $\K$ is, indeed, a
cyclotomic $p$-orientation of $G_{\K}$.
We study profinite groups with a cyclotomic $p$-orientation.
In particular, we show that cyclotomicity is preserved by several operations on profinite groups, and that Bloch-Kato pro-$p$
groups with a cyclotomic $p$-orientation satisfy a strong form of Tits' alternative and decompose as semi-direct product
over a canonical abelian closed normal subgroup.
\end{abstract}

%%% ----------------------------------------------------------------------
\maketitle
%%% ----------------------------------------------------------------------
%\tableofcontents

%%%%%%%%%%

\section{Introduction}
\label{s:intro}

For a prime $p$ let $\Z_p$ denote the ring of $p$-adic integers.
For a profinite group $G$, we call a continuous representation $\theta\colon G\to\Z_p^\times=\mathrm{GL}_1(\Z_p)$
a {\sl $p$-orientation} of $G$ and call the couple $(G,\theta)$ a {\sl $p$-oriented} profinite group.
Given a $p$-oriented profinite group $(G,\theta)$, for $k\in\Z$ let $\Z_p(k)$ denote the left $\Z_p\dbl G\dbr$-module induced by
$\theta^k$, namely, $\Z_p(k)$ is equal to the additive group $\Z_p$ and the left $G$-action is given by
\begin{equation}
\label{eq:defZp1}
g\cdot z=\theta(g)^{k}\cdot z,\qquad g\in G,\ z\in\Z_p(k).
\end{equation}
Vice-versa, if $M$ is a topological left $\Z_p\dbl G\dbr$-module which as an abelian pro-$p$ group is isomorphic to $\Z_p$, then there exists a unique $p$-orientation 
$\theta\colon G\to\Z_p^\times$ such that $M\simeq \Z_p(1)$.

The $\Z_p\dbl G\dbr$-module $\Z_p(1)$ and the representation $\theta\colon G\to\Z_p^\times$ are said to be 
{\sl $k$-cyclotomic}, for $k\geq1$, if for every open subgroup $U$ of $G$ and every $n\geq1$ the natural maps
\begin{equation}\label{eq:cyc}
 \xymatrix{ H^k(U,\Z_p(k)/p^n)\ar[r] & H^k(U,\Z_p(k)/p) },
\end{equation}
induced by the epimorphism of $\Z_p\dbl U\dbr$-modules $\Z_p(k)/p^n\to\Z_p(k)/p$, are surjective.
If $\Z_p(1)$ (respectively $\theta$) is $k$-cyclotomic for every $k\geq1$, then it is called simply a
cyclotomic $\Z_p\dbl G\dbr$-module (resp., cyclotomic $p$-orientation).
Note that $\Z_p(1)$ is $k$-cyclotomic if, and only if, $H^{k+1}_{\cts}(U,\Z_p(k))$ is a torsion free $\Z_p$-module
for every open subgroup $U\subseteq G$ --- here $H_{\cts}^*$ denotes continuous cochain cohomology as introduced
by J.~Tate in \cite{tate:miln} (see \S~\ref{ss:delta}).

Cyclotomic modules of profinite groups have been introduced and studied by C.~De~Clercq and M.~Florence in \cite{smooth}.
Previously J.P.~Labute, in \cite{labute:demushkin}, considered surjectivity of \eqref{eq:cyc} in the case $k=1$ and $U=G$ ---
note that demanding surjectivity for $U=G$ only is much weaker than demanding it for every
open subgroup $U\subseteq G$, and this is what makes the definition of cyclotomic modules truly new.

Let $\K$ be a field, and let $\bar{\K}/\K$ be a separable closure of $\K$.
If $\chr(\K)\not=p$, the {\sl absolute Galois group} $G_{\K}=\Gal(\bar{\K}/\K)$ of $\K$
comes equipped with a canonical $p$-orientation
\begin{equation}
\label{eq:char}
\theta_{\K,p}\colon G_{\K}\longrightarrow \Aut(\mu_{p^\infty}(\bar{\K}))\simeq\Z_p^\times,
\end{equation}
where $\mu_{p^\infty}(\bar{\K})\subseteq \bar{\K}^\times$ 
denotes the subgroup of roots of unity of $\bar{\K}$ of $p$-power order.
If $p=\chr(\K)$, we put $\theta_{\K,p}= \eqo_{G_{\K}}$, the function which is constantly $1$ on $G_{\K}$.
The following result (cf. \cite[Prop.~14.19]{smooth}) is a consequence of the positive solution
of the Bloch-Kato Conjecture given by M.~Rost and V.~Voevodsky with the ``C.~Weibel patch''
(cf. \cite{rost:BK,voev,weibel}), which from now on we will refer to as the Rost-Voevodsky Theorem.

\begin{thm}\label{thmA}
 Let $\K$ be a field, and let $p$ be prime number.
The canonical $p$-orientation $\theta_{\K,p}\colon G_{\K}\to \Z_p^\times$ is cyclotomic.
\end{thm}

Theorem~\ref{thmA} provides a fundamental class of examples of profinite groups
endowed with a cyclotomic $p$-orientation.
Bearing in mind the exotic character of absolute Galois groups, it also provides
a strong motivation to the study of cyclotomically $p$-oriented profinite groups
--- which is the main purpose of this manuscript.
In fact, one may recover several Galois-theoretic statements already for profinite groups with
a 1-cyclotomic $p$-orientation
--- e.g., the only finite group endowed with a 1-cyclotomic $p$-orientation is the finite group $C_2$ of order 2,
with non-constant 2-orientation $\theta\colon C_2\to\{\pm1\}$ (cf. \cite[Ex.~3.5]{eq:kummer}), 
and this implies the Artin-Schreier obstruction for absolute Galois groups.
In their paper, De~Clercq and Florence formulated the ``Smoothness Conjecture'',
which can be restated in this context as follows: for a $p$-oriented profinite group,
1-cyclotomicity implies $k$-cyclotomicity for all $k\geq1$ (cf. \cite[Conj.~14.25]{smooth}).

A $p$-oriented profinite group $(G,\theta)$ is said to be {\sl Bloch-Kato} if the $\F_p$-algebra
\begin{equation}\label{eq:H}
 H^\bullet(U,\htheta\vert_U)=\coprod_{k\geq0} H^k(U,\F_p(k)),
\end{equation}
where $\F_p(k)=\Z_p(k)/p$, with product given by cup-product,
is quadratic for every open subgroup $U$ of $G$.
Note that if $\image(\theta)\subseteq1+p\Z_p$ and $p\neq2$
then $G$ acts trivially on $\Z_p(k)/p$.
By the Rost-Voevodsky Theorem $(G_{\K},\theta_{\K,p})$ is, indeed, Bloch-Kato.

%In the class of $p$-oriented profinite (or pro-$p$) groups one has three operations:
%inverse limits, free profinite (resp. pro-$p$) product of two $p$-oriented profinite (resp. pro-$p$) groups,
%and the fibre product $A\rtimes (G,\theta)$ of a $p$-oriented profinite group $(G,\theta)$ with a free abelian pro-$p$ group 
%$A$, with $gag^{-1}=a^{\theta(g)}$ for all $a\in A$ and $g\in G$ (cf. \S~\ref{ss:fibprod}).
%Moreover, given a $p$-oriented profinite group $(G,\theta)$, one may consider the $p$-oriented quotient $(G/N,\theta)$
%with $N$ a closed normal subgroup of $G$ contained in $\kernel(\theta)$.
%These operations preserve cyclotomicity.
For a profinite group $G$, let $O_p(G)$ denote the maximal closed normal pro-$p$ subgroup of $G$.
A $p$-oriented profinite group $(G,\theta)$ has two particular closed normal subgroups:
the kernel $\kernel(\theta)$ of $\theta$, and the {\sl $\theta$-center} of $(G,\theta)$, given by
\begin{equation}
\label{eq:thetaZ}
\Zen_\theta(G)=\left\{\,x\in O_p(\kernel(\theta))\mid gxg^{-1}=x^{\theta(g)}\ \text{for all $g\in G$}\,\right\}.
\end{equation}
As $\Zen_\theta(G)$ is contained in the center $\Zen(\kernel(\theta))$ of $\kernel(\theta)$, it is abelian.
The $p$-oriented profinite group $(G,\theta)$ will be said to be 
{\sl $\theta$-abelian}, if $\kernel(\theta)=\Zen_\theta(G)$ and if $\Zen_\theta(G)$ is torsion free. 
In particular, for such a $p$-oriented profinite
group $(G,\theta)$, $G$ is a virtual pro-$p$ group (i.e., $G$ contains an open subgroup which is a pro-$p$ group).
Moreover, a $\theta$-abelian pro-$p$ group $(G,\theta)$ will be said to be {\sl split} if
$G\simeq\Zen_\theta(G)\rtimes\image(\theta)$.

As $\Zen_\theta(G)$ is contained in $\kernel(\theta)$, by definition, the canonical quotient
$\baG=G/\Zen_\theta(G)$ carries naturally a $p$-orientation $\btheta\colon\baG\to
\Z_p^\times$, and one has the following short exact sequence of $p$-oriented profinite groups.

\begin{equation}\label{eq:spq}
\xymatrix{ \triv\ar[r]& \Zen_\theta(G)\ar[r]&G\ar[r]^-\pi&\baG\ar[r]&\triv }
\end{equation}

The following result can be seen as an analogue of the equal characteristic transition theorem 
(cf. \cite[\S II.4, Exercise~1(b), p.~86]{ser:gal}) for cyclotomically $p$-oriented Bloch-Kato profinite groups.

\begin{thm}\label{thmC}
Let $(G,\theta)$ be a cyclotomically $p$-oriented Bloch-Kato profinite group. Then
\eqref{eq:spq} splits, provided that $\ccd_p(G)<\infty$, and one of the following conditions hold:
\begin{itemize}
\item[\textup{(i)}] $G$ is a pro-$p$ group,
\item[\textup{(ii)}] $(G,\theta)$ is an oriented virtual pro-$p$ group \textup{(}see \S 4
\textup{)},
\item[\textup{(iii)}] $(\baG,\btheta)$ is cyclotomically $p$-oriented and Bloch-Kato.
\end{itemize}
\end{thm}

In the case that $(G,\theta)$ is the maximal pro-$p$ Galois group of a field $\K$ containing a primitive $p^{th}$-root of unity
endowed with the $p$-orientation induced by $\theta_{\K,p}$, $\Zen_\theta(G)$
is the inertia group of the maximal $p$-henselian valuation of $\K$ (cf. Remark~\ref{rem:arithmetic_Z}). 

Note that the 2-oriented pro-2 group $(C_2\times\Z_2,\theta)$ may be $\theta$-abelian, 
but $\theta$ is never 1-cyclotomic (cf. Proposition~\ref{fact:C2Z2}).
As a consequence, in a cyclotomically 2-oriented pro-$2$ group every element of order 2 is self-centralizing.

For $p$ odd it was shown in \cite{cq:BK}
that a Bloch-Kato pro-$p$ group $G$ satisfies a strong form of {\sl Tits alternative},
i.e., either $G$ contains a closed non-abelian free pro-$p$ subgroup, or
there exists a $p$-orientation $\theta\colon G\to\Z_p^\times$ such that $G$ is $\theta$-abelian.
In Subsection~\ref{ss:2gen} we extend this result to pro-$2$ groups with a cyclotomic orientation, i.e.,
one has the following analogue of R.~Ware's theorem (cf. \cite{ware}) for 
cyclotomically oriented Bloch-Kato pro-$p$ groups (cf. Fact~\ref{fact:tits}).

\begin{thm}\label{thmD}
 Let $(G,\theta)$ be a cyclotomically $p$-oriented Bloch-Kato pro-$p$ group.
 If $p=2$ assume further that $\image(\theta)\subseteq 1+4\Z_2$.
Then one --- and only one --- of the following cases hold:
\begin{itemize}
\item[\textup{(}i\textup{)}] $G$ contains a closed non-abelian free pro-$p$ subgroup; or
\item[\textup{(}ii\textup{)}] $G$ is $\theta$-abelian.
\end{itemize}
\end{thm}

It should be mentioned that for $p=2$ the additional hypothesis is indeed necessary
(cf. Remark~\ref{rem:kb2}).
The class of cyclotomically $p$-oriented Bloch-Kato profinite groups is closed with respect to several constructions.

\begin{thm}\label{thmB}
\begin{itemize}
\item[\textup{(}a\textup{)}] The inverse limit of an inverse system of cyclotomically $p$-oriented Bloch-Kato profinite groups
with surjective structure maps 
is a cyclotomically $p$-oriented Bloch-Kato profinite group (cf. Corollary~\ref{cor:comp2} and Corollary~\ref{cor:quadlim}).
\item[\textup{(}b\textup{)}] The free profinite (resp. pro-$p$) product of two cyclotomically $p$-oriented Bloch-Kato profinite (resp. pro-$p$) groups
is a cyclotomically $p$-oriented Bloch-Kato profinite (resp. pro-$p$) group (cf. Theorem~\ref{thm:freeprod}).
\item[\textup{(}c\textup{)}] The fibre product of a cyclotomically $p$-oriented Bloch-Kato profinite group $(G_1,\theta_1)$
with a split $\theta_2$-abelian profinite group $(G_2,\theta_2)$
is a cyclotomically $p$-oriented Bloch-Kato profinite group (cf. Theorem~\ref{thm:cycfib} and Theorem~\ref{thm:fibco}).
\item[\textup{(}d\textup{)}] The quotient of a cyclotomically $p$-oriented Bloch-Kato profinite group $(G,\theta)$
with respect to a closed normal subgroup $N\subseteq G$ satisfying $N\subseteq\kernel(\theta)$ and $N$ a $p$-perfect group
is a cyclotomically $p$-oriented Bloch-Kato profinite group (cf. Proposition~\ref{propB}).% --- in particular, the maximal pro-$p$ quotient
%of a cyclotomically $p$-oriented Bloch-Kato profinite group is a cyclotomically $p$-oriented Bloch-Kato pro-$p$ group.
\end{itemize}
\end{thm}

Some time ago I.~Efrat (cf. \cite{ido:ETC,efrat:ETC,ido:small}) has formulated the so-called
{\sl elementary type conjecture} concerning the structure of finitely generated pro-$p$ groups 
occurring as maximal pro-$p$ quotients of an absolute Galois group.
His conjecture can be reformulated in the class of cyclotomically $p$-oriented
Bloch-Kato pro-$p$ groups.
Such a $p$-oriented pro-$p$ group $(G,\theta)$ is said to be
{\sl indecomposable} if $\Zen_\theta(G)=\triv$ and if $G$ is not a proper free 
pro-$p$ product.
A positive answer to the following question would settle the elementary type conjecture
affirmatively.

\begin{ques}\label{ques:ETC}
Let $(G,\theta)$ be a finitely generated, torsion free, indecomposable, cyclotomically oriented
Bloch-Kato pro-$p$ group. Does this imply that $G$ is a Poincar\'e duality pro-$p$ group of dimension $\ccd_p(G)\leq 2$?
\end{ques}

{\small
The paper is organized as follows.
In \S~\ref{s:absgal} we give some equivalent definitions for cyclotomic $p$-orientations.
In \S~\ref{s:cpo} we study some operations of profinite groups (inverse limits, free products and fibre products) in relation
with the properties of cyclotomicity and Bloch-Kato-ness, and we prove Theorem~\ref{thmB}(a)-(b)-(c).
In \S~\ref{s:virtual} we study the quotients of cyclotomically $p$-oriented profinite groups over closed normal $p$-perfect subgroups
--- in particular, we introduce {\sl oriented virtual pro-$p$ groups} and we prove Theorem~\ref{thmB}(d).
In \S~\ref{s:PDgps} we study $p$-oriented profinite Poincar\'e duality groups.
In \S~\ref{s:torBK} we focus on the presence of torsion in cyclotomically 2-oriented pro-2 groups, and we prove that 
in a 1-cyclotomically 2-oriented pro-2 group every element of order 2 is self-centralizing (see Proposition~\ref{fact:C2Z2}).
In \S~\ref{s:prop} we focus on the structure of cyclotomically $p$-oriented Bloch-Kato pro-$p$ groups:
we prove Theorems~\ref{thmC} and \ref{thmD}, and show that in many cases the $\theta$-center is the maximal abelian closed normal subgroup
(cf. Theorem~\ref{prop:Zmxlabeliannormal}).}

%%%%%%%%%%%%%%%%%%%%%%%%%%%%%%%%%%%%%%%%%%%%%%%%%%%%%%%%%%%%%%%%%%%%%%%%%%%%%%%%%%%%%%%%%%%%%%%%%%%%%%
%%%%%%%%%%%%%%%%%%%%%%%%%%%%
%%%%%%%%%%%%%%%%%%%%%%%%%%%%%%%%%%%%%%%%%%%%%%%%%%%%%%%%%%%%%%%%%%%%%%%%%%%%%%%%%%%%%%%%%%%%%%%%%%%%%%

\section{Absolute Galois groups and cyclotomic $p$-orientations}
\label{s:absgal}
%Throughout this paper subgroups of profinite groups are assumed to be closed (in the profinite topology).

Throughout the paper, we study profinite groups with a cyclotomic module $\Z_p(1)$.
In contrast to \cite[\S~14]{smooth}, we refer to the associated representation $\theta\colon G\to\Z_p^\times$,
rather than to the module itself.
As we study several subgroups of $G$
associated to this cyclotomic module $\Z_p(1)$, like $\kernel(\theta)$ and $\Zen_\theta(G)$,
this choice of notation turns out to be convenient for our purposes.
We follow the convention as established in \cite{cq:BK,cq:phd} and call such representations ``$p$-orientations''.\footnote{
For a Poincar\'e duality group $G$ the representation associated to the {\sl dualizing module} --- which coincides with 
the cyclotomic module in the case of a Poincar\'e duality pro-$p$ group of dimension 2 (cf. Theorem~\ref{thm:propPD2}) ---
is sometimes also called the ``orientation'' of $G$.}
In the case that $G$ is a pro-$p$ group, the couple $(G,\theta)$ was called a {\sl cyclotomic pro-$p$ pair},
in \cite[\S~3]{ido:small}.

%%%%%%%%%%%%%555

\subsection{The connecting homomorphism $\delta^k$}
\label{ss:delta}
Let $G$ be a profinite group, and let $\theta\colon G\to\Z_p^\times$ be a $p$-orientation of $G$.
For every $k\geq0$ one has the short exact sequence of left $\ZpG$-modules
\begin{equation}
\label{eq:sesZpZpFp}
\xymatrix{0\ar[r] & \Z_p(k)\ar[r]^-{p\cdot} & \Z_p(k)\ar[r] & \F_p(k)\ar[r] & 0},
\end{equation}
which induces the long exact sequence in cohomology
\begin{equation}
\label{eq:lesZpZpFp} 
\xymatrix@C=0.8truecm{ \cdots\ar[r]^-{p\cdot} & H_{\cts}^k(G,\Z_p(k)) \ar[r]^-{\pi^k} & H^k(G,\F_p(k))\ar`r[d]`[l]^{\delta^k} `[dll] `[dl] [dl]   \\
& H_{\cts}^{k+1}(G,\Z_p(k))\ar[r]^-{p} & H_{\cts}^{k+1}(G,\Z_p(k))\ar[r] & \cdots}
\end{equation}
with connecting homomorphism $\delta^k$ (cf. \cite[\S 2]{tate:miln}).
In particular, $\delta^k$ is trivial if, and only if, multiplication by $p$ on $H_{\cts}^{k+1}(G,\Z_p(k))$
is a monomorphism. This is equivalent to $H_{\cts}^{k+1}(G,\Z_p(k))$ being torsion free.
Therefore, one concludes the following:

\begin{prop}\label{fact:p torsion}
Let $(G,\theta)$ be a $p$-oriented profinite group.
For $k\geq1$ and $U\subseteq G$ an open subgroup the following are equivalent.
\begin{itemize}
 \item[(i)] The map \eqref{eq:cyc} is surjective for every $n\geq1$.
 \item[(ii)] The map $\pi^k\colon H_{\cts}^k(U,\Z_p(k)) \to H^k(U,\F_p(k))$ is surjective.
 \item[(iii)] The connecting homomorphism $\delta^k\colon H^k(U,\F_p(k))\to H_{\cts}^{k+1}(U,\Z_p(k))$ is trivial.
 \item[(iv)] The $\Z_p$-module $H_{\cts}^{k+1}(U,\Z_p(k))$ is torsion free.
\end{itemize}
\end{prop}

\begin{proof}
By the long exact sequence \eqref{eq:lesZpZpFp}, the equivalences between (ii), (iii) and (iv) are straightforward.
For $m\geq n\geq1$ let $\pi_{m,n}^k$ denote the natural maps
$$\pi^k_{m,n}\colon H^k(U,\Z_p(k)/p^m)\longrightarrow H^k(U,\Z_p(k)/p^n)$$
(if $m=\infty$ we set $p^\infty=0$).
If condition (i) holds then the system $$(H^k(U,\Z_p/p^{n}),\pi_{m,n}^k)$$ satisfies the Mittag-Leffler property.
In particular, $$H^k(U,\Z_p(k))\simeq\varprojlim_{n\geq1}H^k(U,\Z_p(k)/p^n)$$
(cf. \cite{roos} and \cite[Thm.~2.7.5]{nsw:cohn}).
Thus $\pi^k=\pi_{n,1}^k\circ\pi^k_{\infty,n}$ is surjective if, and only if, $\pi_{n,1}^k$ is 
surjective for every $n\geq1$. 
Conversely, if $\pi^k$ is surjective then $\pi^k=\pi_{n,1}^k\circ\pi^k_{\infty,n}$ yields the surjectivity
of $\pi_{n,1}^k$ for every $n$.
\end{proof}

%%%%%%%%%%%%%%%%%%%%%%%%%%%%%%%%%%%%%%%%%%%%%%%%%%%%%%%%%%%%%%%%%%%%

\subsection{Profinite groups of cohomological $p$-dimension at most $1$}
\label{ss:1dim}
Let $G$ be a profinite group, and let $\theta\colon G\to\Z_p^\times$ be a $p$-orientation of $G$.
Then
\begin{equation}
\label{eq:cyc1}
H^1_{\cts}(G,\Z_p(0))=\Hom_{\grp}(G,\Z_p)
\end{equation}
is a torsion free abelian group for every profinite group $G$, i.e., $\theta$ is $0$-cyclotomic.
If $G$ is of cohomological $p$-dimension less or equal to $1$,
then $H^{m+1}_{\cts}(G,\Z_p(m))=0$ for all $m\geq 1$ showing that $\theta$ is cyclotomic.
Moreover, $H^\bullet(G,\hat\theta)$ is a quadratic $\F_p$-algebra for every
profinite group with $\ccd_p(G)\leq 1$ and for any $p$-orientation $\theta\colon G\to\Z_p^\times$.
If $G$ is of cohomological $p$-dimension less or equal to $1$,
one has $\ccd_p(C)\leq1$ for every closed subgroup $C$ of $G$ (cf. \cite[\S I.3.3, Proposition~14]{ser:gal}).
Thus one has the following.

\begin{fact}
\label{fact:ccd1}
Let $G$ be a profinite group with $\ccd_p(G)\leq 1$, and let $\theta\colon G\to\Z_p^\times$
be a $p$-orientation for $G$. Then $(G,\theta)$ is Bloch-Kato and $\theta$ is cyclotomic.
\end{fact}

%%%%%%%%%%%%%%%%%%%%%%%%%%%%%%%%%%%%%%%%%%%%%%%%%%%%%%%%%%%%%%%%%%%%%%%%%%%%%%%%%5

\subsection{The $m^{th}$-norm residue symbol}
\label{ss:tate}
Throughout this subsection we fix a field $\K$ and a separable closure $\bar{\K}$ of $\K$.
For $p\not=\chr(\K)$, $\mu_{p^\infty}(\bar{\K})$ is a {\sl divisible} abelian group.
By construction, one has a canonical isomorphism
\begin{equation}
\label{eq:Qcyc}
\textstyle{\varprojlim_{k\geq 0} (\mu_{p^\infty}(\bar{\K}), p^k)\simeq\Z_p(1)\otimes_{\Z}\Q_p=\Q_p(1)}
\end{equation}
and a short exact sequence $0\to\Z_p(1)\to\Q_p(1)\to\mu_{p^\infty}(\bar{\K})\to 0$
of topological left $\ZpGK$-modules, where $G_{\K}=\Gal(\bar{\K}/\K)$ is the absolute Galois group of $\K$.

Let $K^M_m(\K)$, $m\geq 0$, denote the {\sl $m^{th}$-Milnor $K$-group} of $\K$ (cf. \cite[\S24.3]{ido:miln}).
For $p\not=\chr(\K)$,
J.~Tate constructed in \cite{tate:miln} a homomorphism of abelian groups
\begin{equation}
\label{eq:tate1}
h_m(\K)\colon K^M_m(\K)\longrightarrow H^m_{\cts}(G_{\K},\Z_p(m)),
\end{equation}
the so-called {\sl $m^{th}$-norm residue symbol}.
Let $K^M_m(\K)_{/p}=K^M_m(\K)/pK^M_m(\K)$.
Around ten years later S.~Bloch and K.~Kato conjectured  in \cite{bk:ihes} that the induced map
\begin{equation}
\label{eq:tate2}
h_m(\K)_{/p}\colon K^M_m(\K)_{/p}\longrightarrow H^m(G_{\K},\F_p(m))
\end{equation}
is an isomorphism for all fields $\K$, $\chr(\K)\not=p$, and for all $m\geq 0$.
This conjecture has been proved by V.~Voevodsky and M.~Rost with a ``patch'' of C.~Weibel
(cf. \cite{rost:BK,voev,weibel}).
In particular, since $K^M_\bullet(\K)_{/p}$ is a quadratic $\F_p$-algebra
and as $h_\bullet(\K)_{/p}$ is a homomorphism of algebras, 
this implies that the absolute Galois group of a field $\K$ is Bloch-Kato
(cf. \cite[\S23.4]{ido:miln}).
The Rost-Voevodsky Theorem has also the following consequence.

\begin{prop}
\label{thm:cyc}
Let $\K$ be a field, let $G_{\K}$ denote its absolute Galois group,
and let $\theta_{\K,p}\colon G_{\K}\to\Z_p^\times$ denote its canonical $p$-orientation.
Then $\theta_{\K,p}$ is cyclotomic.
\end{prop}

Although Proposition~\ref{thm:cyc} might be well known to specialists, we add a short proof of it.
By Proposition~\ref{fact:p torsion}, Proposition~\ref{thm:cyc} in combination with Theorem~\ref{thmB}-(d)
is equivalent to \cite[Prop.~14.19]{smooth}.

\begin{proof}[Proof of Proposition~\ref{thm:cyc}]
If $\chr(\K)=p$, then $\ccd_p(G_{\K})\leq 1$ (cf. \cite[\S II.2.2, Proposition~3]{ser:gal}),
and the $p$-orientation $\theta_{\K,p}$ is cyclotomic by Fact~\ref{fact:ccd1}.
So we may assume that $\chr(\K)\neq p$.
In the commutative diagram
\begin{equation}
\label{eq:tate3}
\xymatrix@C=0.4cm{
K^M_k(\K)\ar[r]^{p}\ar[d]^{h_k}&K^M_k(\K)\ar[d]^{h_k}\ar[r]^{\pi}&K^M_k(\K)_{/p}\ar[r]\ar[d]^{(h_{k})_{/p}}&0\\
H^k_{\cts}(G_{\K},\Z_p(k))\ar[r]^{p}&H^k_{\cts}(G_{\K},\Z_p(k))\ar[r]^{\alpha}&H^k(G_{\K},\F_p(k))\ar[r]^-{\beta}&
H^{k+1}_{\cts}(G_{\K},\Z_p(k))
}
\end{equation}
the map $\pi$ is surjective, and $(h_{k})_{/p}$ is an isomorphism.
Hence $\alpha$ must be surjective, and thus $\beta=0$, i.e.,
$$p\colon H^{k+1}_{\cts}(G_{\K},\Z_p(k))\longrightarrow H^{k+1}_{\cts}(G_{\K},\Z_p(k))$$
is an injective homomorphism of $\Z_p$-modules.
Thus $H^{k+1}_{\cts}(G_{\K},\Z_p(k))$ must be $p$-torsion free.
Any open subgroup $U$ of $G_{\K}$ is the absolute Galois group of $\bar{\K}^U$.
Hence $\theta_{\K,p}$ is cyclotomic, and this yields the claim.
\end{proof}

\begin{rem}\label{rem:GKs}
Let $\K$ be a number field, let $S$ be a set of places containing all infinite places of $\K$ and all
places lying above $p$, and let $G_{\K}^S$ be the Galois group of $\bK^S/\K$,
where $\bK^S/\K$ is the maximal extension of $\bK/\K$ which is unramified outside $S$. 
Then $\theta_{\K,p}\colon G_{\K}\to\Z_p^\times$ induces a 
$p$-orientation $\theta_{k,p}^S\colon G_{\K}^S\to\Z_p^\times$.
However, it is well known (cf. \cite[Prop.~8.3.11(ii)]{nsw:cohn}) that,
\begin{equation}
\label{eq:GKS1}
H^1(G_{\K}^S,\II_p(1))\simeq H^1(G_{\K}^S,\caO_{\bK}^S)_{(p)}\simeq\ccl(\caO^S_{\K})_{(p)}
\end{equation} 
(for the definition of $\II_p(1)$ see \S\ref{s:cpo}), where $\ccl(\caO_{\K}^S)$ denotes the {\sl ideal class group} 
of the Dedekind domain $\caO_{\K}^S$, and $\argu_{(p)}$ denotes the $p$-primary component.
Hence $(G_{\K}^S,\theta_{\K,p}^S)$
is in general not cyclotomic (cf. Proposition~\ref{prop:comp}).
\end{rem}

%%%%%%%%%%%%%%%%%%%%%%%%%%%%%%%%%%%%%%%%%%%%%%%%%%%%%%%%%%%%%%%%%%%%%%%%%%%%%%%%%%%%%%%%%%%%%%%%%%
%%%%%%%%%%%%%%
%%%%%%%%%%%%%%%%%%%%%%%%%%%%%%%%%%%%%%%%%%%%%%%%%%%%%%%%%%%%%%%%%%5%%%%%%%%%%%%%%%%%%%%%%%%%%%%%%%%

\section{Cohomology of $p$-oriented profinite groups}
\label{s:cpo}

A homomorphism $\phi\colon (G_1,\theta_1)\to(G_2,\theta_2)$
of two $p$-oriented profinite groups $(G_1,\theta_1)$ and $(G_2,\theta_2)$ is a continuous group homomorphism 
$\phi\colon G_1\to G_2$ satisfying $\theta_1=\theta_2\circ\phi$.

Let $(G,\theta)$ be a $p$-oriented profinite group. 
For $k\in\Z$, put $\Q_p(k)=\Z_p(k)\otimes_{\Z_p}\Q_p$,
and also $\II_p(k)=\Q_p(k)/\Z_p(k)$, i.e., $\II_p(k)$ is a discrete left $G$-module
and --- as an abelian group --- a divisible $p$-torsion module.

Let $\II_p=\Q_p/\Z_p$, and let $\argu^\ast=\Hom_{\Z_p}(\argu,\II_p)$ denote the Pontryagin duality functor.
Then $\II_p(k)^\ast$ is a profinite left $\ZpG$-module which is isomorphic to $\Z_p(-k)$.

%%%%%%%%%%%%%%%%%%%%%%%%%%%%%%%%%%%%%%%%%%%%%%%%%%%%%%%%%%%%%%%%%%%%%%%%5

\subsection{Criteria for cyclotomicity}
\label{ss:cyclocrit}
The following proposition relates the continuous co-chain cohomology groups,
Galois cohomology and the Galois homology groups as defined by A.~Brumer in \cite{brum:pseudo}.

\begin{prop}
\label{prop:comp}
Let $(G,\theta)$ be a $p$-oriented profinite group, let $k$ be an integer, and let $m$ be a non-negative integer. 
Then the following are equivalent:
\begin{itemize}
\item[(i)] $H^{m+1}_{\cts}(G,\Z_p(k))$ is torsion free;
\item[(ii)] $H^m(G,\II_p(k))$ is divisible;
\item[(iii)] $H_m(G,\Z_p(-k))$ is torsion free.
\end{itemize}
\end{prop}

\begin{proof}
The equivalence (i)$\Leftrightarrow$(ii)  is a direct consequence of \cite[Prop.~2.3]{tate:miln},
and (ii)$\Leftrightarrow$(iii) follows from \cite[(3.4.5)]{sw:cpg}.
\end{proof} 

The direct limit of divisible $p$-torsion modules is a divisible $p$-torsion module.
From this fact --- and Proposition~\ref{prop:comp} --- one concludes the following.

\begin{cor}
\label{cor:comp}
Let $(G,\theta)$ be a cyclotomically $p$-oriented profinite group.
Then $H^m(C,\II_p(m))$ is divisible for all $m\geq 0$ and all $C$ closed in $G$.
\end{cor}

\begin{proof}
It suffices to show (ii)$\Rightarrow$(i). Let $C$ be a closed subgroup
of $G$. Then $$H^m(C,\II_p(m))\simeq \varinjlim_{U\in\euB_C} H^m(U,\II_p(m)),$$
where $\euB_C$ denotes the set of all open subgroups of $G$ containing $C$ 
(cf. \cite[\S I.2.2, Proposition~8]{ser:gal}). Hence Proposition~\ref{prop:comp} yields the claim.
\end{proof}

In combination with \cite[Corollary~4.3(ii)]{brum:pseudo}, Proposition~\ref{prop:comp} implies the following.

\begin{cor}
\label{cor:comp2}
Let $(I,\preceq)$ be a directed set, 
let $(G,\theta)$ be a $p$-oriented profinite group, and let $(N_i)_{i\in I}$
be a family of closed normal subgroups of $G$ satisfying $N_j\subseteq N_i\subseteq\kernel(\theta)$ for 
$i\preceq j$ such that $\bigcap_{i\in I} N_i=\triv$ and
the induced $p$-orientation $\theta_i\colon G/N_i\to\Z_p^\times$ is cyclotomic
for all $i\in I$.
Then $\theta\colon G\to\Z_p^\times$ is cyclotomic.
\end{cor}

\begin{proof}
Let $U\subseteq G$ be a open subgroup of $G$.
Hypothesis (iii) implies that the group $H_m(UN_i/N_i,\Z_p(-m))$ is torsion free 
for all $i\in I$ (cf. Proposition~\ref{prop:comp}). Thus, by
\cite[Corollary~4.3(ii)]{brum:pseudo}, $H_m(U,\Z_p(-m))$ is torsion free,
and hence, by Proposition~\ref{prop:comp}, 
$\theta\colon G\to\Z_p^\times$ is a cyclotomic $p$-orientation.
\end{proof}

%%%%%%%%%%%%%%%%%%%%%%%%%%%%%%%%%%%%%%%%%%%%%%%%%%%%%%%%%%%%%%%%%%%%%$$$$$$$$$$$

\subsection{The mod-$p$ cohomology ring}

\label{ss:modpco}
An $\N_0$-graded $\F_p$-algebra $\bA=\coprod_{k\geq 0}\bA_k$ is said to be
{\sl anti-commutative} if for $x\in\bA_s$ and $y\in\bA_t$ one has $y\cdot x=(-1)^{st}\cdot x\cdot y$.
E.g., if $V$ is an $\F_p$-vector space, the {\sl exterior algebra}
$\Lambda_\bullet(V)$ (cf. \cite[Chapter~4]{mcl:hom}) is an $\N_0$-graded anti-commutative $\F_p$-algebra.
Moreover, if $G$ is a profinite group, then its cohomology ring $H^\bullet(G,\F_p)$
is an $\N_0$-graded anti-commutative $\F_p$-algebra (cf. \cite[Prop.~1.4.4]{nsw:cohn}).

Let $\boT(V)=\coprod_{k\geq 0}V^{\otimes k}$ denote the {\sl tensor algebra} generated by the $\F_p$-vector space $V$.
A $\N_0$-graded associative $\F_p$-algebra $\bA$ is said to be {\sl quadratic} if the 
canonical homomorphism $\eta^{\bA}\colon \boT(\bA_1)\to\bA$ is surjective, and 
\begin{equation}
\label{eq:quad}
\kernel(\eta^{\bA})=\boT(\bA_1)\otimes
\kernel(\eta_2^\bA)\otimes\boT(\bA_1)
\end{equation}
(cf. \cite[\S~1.2]{pp:quad}).
E.g., $\bA=\Lambda_\bullet(V)$ is quadratic.

If $\bA$ and $\bB$ are anti-commutative $\N_0$-graded $\F_p$-algebras,
then $\bA\otimes\bB$ is again an anti-commutative $\N_0$-graded $\F_p$-algebra,
where 
\begin{equation}
\label{eq:cori2}
(x_1\otimes y_1)\cdot(x_2\otimes y_2)=(-1)^{s_2t_1}\cdot 
(x_1\cdot x_2)\otimes (y_1\cdot y_2),
\end{equation}
for $x_1\in\bA_{s_1},\,x_2\in\bA_{s_2}\,y_1\in\bB_{t_1},\,y_2\in\bB_{t_2}$.
In particular, if $\bA$ and $\bB$ are quadratic, then $\bA\otimes\bB$ is quadratic
as well.

A direct set $(I,\preceq)$ maybe considered as a small category
with objects given by the set $I$ and precisely one morphism
$\iota_{i,j}$ for all $i\preceq j$, $i,j\in I$, i.e., $\iota_{i,i}=\iid_i$.
One has the following.

\begin{fact}
\label{prop:dirlimquad}
Let $\F$ be a field, let $(I,\preceq)$ be a direct system, and let $\bA\colon (I,\preceq)\to\FQAlg$ be a covariant functor
with values in the category of quadratic $\F$-algebras. 
Then $\bB=\varinjlim_{i\in \bA}\bA(i)$ is a quadratic
$\F$-algebra.
\end{fact}

%\begin{proof}
%By hypothesis, one has a short exact sequence of $\F$-vector spaces
%\begin{equation}
%\label{eq:dirquad1}
%\xymatrix@C=0.4cm{
%0\ar[r]&\boT(\bA_1(i))\otimes R(i)\otimes
%\boT(\bA_1(i))
%\ar[r]&\boT(\bA_1(i))\ar[r]&
%\bA(i)\ar[r]&0
%}
%\end{equation}
%for all $i\in I$, where $R(i)=\kernel(\eta_2^{\bA(i)})\subseteq \bA_1(i)\otimes\bA_1(i)$.
%Here we considered the left-hand side of \eqref{eq:dirquad1} as the ideal generated by
%$R(i)$. In particular,
%\begin{equation}
%\label{eq:dirquad2}
%\xymatrix{
%0\ar[r]&R(i)
%\ar[r]&\bA_1(i)\otimes\bA_1(i)\ar[r]&
%\bA_2(i)\ar[r]&0 }
%\end{equation}
%is exact. Moreover, $\varinjlim_{i\in I}$ is an exact functor commuting with 
%tensor products (in both arguments) and thus also for $\boT_k(\argu)$ for all $k\geq 1$.
%Hence it commutes also with the tensor algebra functor $\boT(\argu)$.
%Thus, applying $\varinjlim_{i\in I}$ to the exact sequences \eqref{eq:dirquad2} yields the exact sequence
%\begin{equation}
%\label{eq:dirquad3}
%\xymatrix{
%0\ar[r]&\boT(\bB)\otimes R\otimes
%\boT(B)
%\ar[r]&\boT(\bB_1)\ar[r]&
%\bB\ar[r]&0
%}
%\end{equation}
%for $R=\varinjlim_{i\in I} R(i)\subseteq \bB_1\otimes\bB_1$, and thus the claim.
%\end{proof}

Let $(G,\theta)$ be a $p$-oriented profinite group, and let $\htheta\colon G\to\F_p^\times$
be the map induced by $\theta$. If $\htheta=\eqo_G$, then the {\sl mod-$p$ cohomology ring}
of $H^\bullet(G,\htheta)$ coincides with $H^\bullet(G,\F_p)$ (see \eqref{eq:H}),
and hence it is anti-commutative. Furthermore, if 
$\htheta\not=\eqo_G$ and $G^\circ=\kernel(\htheta)$, restriction 
\begin{equation}\label{eq:rst_virtual}
 \rst^\bullet_{G,G^\circ}\colon H^\bullet(G,\hat\theta)\longrightarrow H^\bullet(G^\circ,\F_p)
\end{equation}
is an injective homomorphism of $\N_0$-graded algebras.
Hence the mod-$p$ cohomology ring $H^\bullet(G,\theta)$ is
anti-commutative. In particular, if $M_{(k)}$ denotes
the homogeneous component of the left $\F_p[G/G^\circ]$-module $M$,
on which $G/G^\circ$ acts by $\htheta^k$, 
the Hochschild-Serre spectral sequence (cf. \cite[\S~II.4, Exercise~4(ii)]{nsw:cohn}) shows that
\begin{equation}
\label{eq:cori5}
H^k(G,\htheta)=H^k(G^\circ,\F_p)_{(-k)}.
\end{equation}
From \cite[\S I.2.2, Prop.~8]{ser:gal} and Fact~\ref{prop:dirlimquad} one concludes the following.

\begin{cor}
\label{cor:Bkop}
Let $(G,\theta)$ be a $p$-oriented profinite group which is Bloch-Kato.
Then $H^\bullet(C,\htheta\vert_C)$ is quadratic for all $C$ closed in $G$.
\end{cor}

\begin{cor}
\label{cor:quadlim}
Let $(I,\preceq)$ be a directed set, 
let $(G,\theta)$ be a $p$-oriented profinite group,
and let $(N_i)_{i\in I}$ be a family of closed normal subgroups of $G$,
$N_j\subseteq N_i\subseteq \kernel(\theta)$ for $i\preceq j$, such that $\bigcap_{i\in I} N_i=\triv$
and $(G/N_i,\htheta_{N_i})$ is Bloch-Kato.
Then $(G,\theta)$ is Bloch-Kato.
\end{cor}

%\begin{proof}
%Let $U$ be an open subgroup of $G$.
%By (iii), $H^\bullet(UN_i/N_i,\htheta_{N_i})$ is quadratic, and (ii) implies that
%$H^\bullet(U,\htheta\vert_U)=\varinjlim_{i\in I} H^\bullet(UN_i/N_i,\htheta_{N_i})$,
%where the mappings are given by inflation (cf. \cite[\S I.2.2, Proposition~8]{ser:gal}).
%Hence Proposition~\ref{prop:dirlimquad} completes the proof.
%\end{proof}

\begin{rem}\label{rem:massey}
Let $G$ be a pro-$p$ group with minimal presentation
 \[
  G=\left\langle\,  x_1,\ldots,x_d\mid[x_1,x_2][[x_3,x_4],x_5]=1 \,\right\rangle,
 \]
with $d\geq5$.
In \cite[Ex.~7.3]{JanTan:massey1} and \cite[\S~4.3]{JanTan:massey2} it is shown that $G$ does not occur
as maximal pro-$p$ Galois group of a field containing a primitive $p^{th}$-root of unity,
relying on the properties of {\sl Massey products}.
It would be interesting to know whether $G$ admits a cyclotomic $p$-orientation $\theta\colon G\to\Z_p^\times$
such that $(G,\theta)$ is Bloch-Kato.
By Theorem~\ref{thmA}, a negative answer would provide a ``Massey-free'' proof of
the aforementioned fact.
\end{rem}

%%%%%%%%%%%%%%%%%%%%%%%%%%%%%%%%%%%%%%%%%%%%%%%%%%%%%%%%%%%%%%%%%%%%%%%%%%%%%%%%%%%%%%

\subsection{Fibre products}
\label{ss:fibprod}

Let $(G_1,\theta_1)$, $(G_2,\theta_2)$ be $p$-oriented profinite groups.
The {\sl fibre product} $(G,\theta)=(G_1,\theta_1)\boxtimes(G_2,\theta_2)$ denotes
the pull-back of the diagram
\begin{equation}
\label{eq:pb}
\xymatrix{
G_1\ar[r]^{\theta_1}&\Z_p^\times\\
G\ar@{-->}[u]\ar@{-->}[r]\ar@{..>}[ur]^{\theta}&G_2\ar[u]_{\theta_2}
}
\end{equation}

\begin{rem}\label{rem:fibprod}
By restricting to the suitable subgroups if necessary, 
for the analysis of a fibre product $(G,\theta)=(G_1,\theta_1)\boxtimes(G_2,\theta_2)$
one may assume that $\image(\theta_1)=\image(\theta_2)$.
In particular, if $(G_2,\theta_2)$ is split $\theta_2$-abelian and $G_2\simeq A\rtimes\image(\theta_2)$ for some free abelian pro-$p$
group $A$, then $G\simeq A\rtimes G_1$ with $gag^{-1}=a^{\theta_1(g)}$ for all $a\in A$ and $g\in G_1$.
\end{rem}

\begin{fact}
\label{fact:HI}
Let $(G,\theta)$ be a $p$-oriented profinite group, and let $N$ be a finitely generated non-trivial
torsion free closed subgroup of $\Zen_{\theta}(G)$,
i.e., $N\simeq\Z_p(1)^r$ as left $\ZpG$-modules for some $r\geq1$.
Then for $k\geq 0$ one has
\begin{equation}
\label{eq:HI}
H^1(N,\II_p(k))\simeq\II_p(k-1)^r
\end{equation}
as left $\ZpG$-module.
\end{fact}

%\begin{proof}
%Since $N\subseteq \kernel(\theta)$, $\II_p(k)$ is a trivial left $\Z_p\dbl N\dbr$-module.
%For $g\in G$, $h\in N$ and $f\in H^1(N,\II_p)=\Hom(N,\II_p)$ one has
%\begin{equation}
%\label{eq:HI2}
%(g\cdot f)(h)=g\cdot f(g^{-1}hg)=\theta(g)^k\cdot f\left(h^{\theta(g)^{-1}}\right)=\theta(g)^{k-1}\cdot f(h),
%\end{equation}
%and this yields the claim.
%\end{proof}

The following property will be useful for the analysis of fibre products.

\begin{lem}
\label{lem:splGr}
Let $(G_1,\theta)$ be a cyclotomically $p$-oriented profinite group, and set
$(G,\theta)=(G_1,\theta_1)\boxtimes (G_2,\theta_2)$,
where $(G_2,\theta_2)$ is split $\theta_2$-abelian with $Z=\Zen_{\theta_2}(G_2)$.
Let $\pi\colon G\to G_1$ be the canonical projection, and let $U\subseteq G$ be an open subgroup.
Then $U\simeq (Z\cap U)\rtimes \pi(U)$.
%In particular, if $Z\cap C\not=\triv$, then $(C,\theta_C)$ has a non-trivial decomposition
%$(C,\theta_C)\simeq (\pi(C),\theta_1\vert_{\pi(C)})\boxtimes (C_2,\btheta_2)$, where $C_2=(Z\cap C)\rtimes\theta(C)$.
\end{lem}

\begin{proof}
Without loss of generality we may assume that $Z\simeq\Z_p$, so that $Z\cap U=Z^{p^k}$ for some $k\geq0$.
It suffices to show that there exists an open subgroup $U_1$ of $U$ satisfying
$Z\cap U_1=\triv$ and $\pi(U_1)=\pi(U)$.

By choosing a section $\sigma\colon G_1\to G$ (see Remark~\ref{rem:fibprod}), one has a continuous 
homomorphism $\tau=\sigma\circ\pi\colon G\to G_1$ and a continuous function
$\eta\colon G\to Z$ such that each
$g\in G$ can be uniquely written as $g=\eta(g)\cdot\tau(g)$.
In particular, for $h,h_1,h_2\in U$ and $z\in Z\cap U=Z^{p^k}$ one has
\begin{equation}
\label{eq:spl1}
\eta(z\cdot h)=z\cdot\eta(h)\qquad
\text{and}\qquad
\eta(h_1\cdot h_2)=\eta(h_1)\cdot{}^{h_1}\eta(h_2).
\end{equation}
Let $\eta_U=\chi\circ\eta\vert_U$, where $\chi\colon Z\to Z/Z^{p^k}$ is the canonical projection.
By \eqref{eq:spl1}, $\eta_U$ defines a crossed-homomorphism $\teta_U\colon \bar U\to Z/Z^{p^k}$, where $\bar U=U/Z^{p^k}$.
As $\bar U$ is canonically isomorphic to an open subgroup of $G_1$, $(\bar U,\theta_1\vert_{\bar U})$ is
cyclotomically $p$-oriented.
(Note that $Z\simeq\Z_p(1)$ as $\Z_p\dbl U\dbr$-modules.)
Hence, $H^1_{\cts}(\bar U,\Z_p(1))\to H^1(\bar U,\Z_p(1)/p^k)$ is surjective by Proposition~\ref{fact:p torsion},
and the snake lemma applied to the commutative diagram
\begin{equation}
\label{eq:spl2}
\xymatrix@C=0.5cm{
0\ar[r]&\mathcal{B}^1(\bar U,Z)\ar[r]\ar@{->>}[d]&\mathcal{Z}^1(\bar U,Z)\ar[r]\ar[d]&H^1(\bar U,\Z_p(1))\ar[r]\ar@{->>}[d]&0\\
0\ar[r]&\mathcal{B}^1(\bar U,Z/Z^{p^k})\ar[r]&\mathcal{Z}^1(\bar U,Z/Z^{p^k})\ar[r]&H^1(\bar U,\Z_p(1)/p^k)\ar[r]&0
}
\end{equation}
where the left-side and right-side vertical arrows are surjective,
shows that $\mathcal{Z}^1(\bar U,Z)\to \mathcal{Z}^1(\bar U,Z/Z^{p^k})$ is surjective.
Thus there exists $\eta_\circ\in \mathcal{Z}^1(\bar U,Z)$ such that $\teta_U=\chi\circ\eta_\circ$.
It is straightforward to verify that
$U_1=\{\eta_\circ(\bh)\cdot\sigma(\bh)\mid \bh\in\bar U\}$
is an open subgroup of $G_1$ satisfying the requirements.
\end{proof}

\begin{thm}
\label{thm:cycfib}
Let $(G_1,\theta_1)$ be a cyclotomically $p$-oriented profinite group,
and let $(G_2,\theta_2)$ be split $\theta_2$-abelian.
Then $(G_1,\theta_1)\boxtimes(G_2,\theta_2)$ is cyclotomically $p$-oriented.
\end{thm}

\begin{rem}
\label{rem:thetab}
{\rm (a)} If $p$ is odd, then every $\theta$-abelian profinite group $(G,\theta)$ is split.
However, a 2-oriented $\theta$-abelian profinite group $(G,\theta)$ is split
 if, and only if, it is cyclotomically $2$-oriented (cf. Proposition~\ref{prop:orC2}).
 
 \noindent
 {\rm (b)}
If $(G,\theta)$ is $\theta$-abelian and $H\subseteq G$ is a closed subgroup, then 
$(H,\theta\vert_H)$ is also $\theta$-abelian.
\end{rem}

\begin{proof}[Proof of Theorem~\ref{thm:cycfib}]
Put $(G,\theta)=(G_1,\theta_1)\boxtimes(G_2,\theta_2)$ and $Z=\Zen_{\theta_2}(G_2)$.
We may also assume that $\image(\theta_1)=\image(\theta_2)$.
As $(G_2,\theta_2)$ is split $\theta_2$-abelian, one has $G=Z\rtimes G_1$. 

We first show the claim for $Z\simeq\Z_p$.
Let $U$ be an open subgroup of $G$.
By Lemma~\ref{lem:splGr}, $(U,\theta\vert_U)\simeq(U_1,\btheta_1)\boxtimes(U_2,\btheta_2)$
where $U_1$ is isomorphic to an open subgroup of $G_1$ and $(U_2,\btheta_2)$ is 
split $\btheta_2$-abelian with $N=\kernel(\btheta_2)$ open in $Z$.
As $\ccd_p(N)=1$, one has $H^m(N,\II_p(k))=0$ for $m\geq2$ and $k\geq0$.
Therefore, the $E_2$-term of the Hochschild-Serre spectral sequence associated to the short exact sequence
of profinite groups 
\begin{equation}
\label{eq:spl4}
\xymatrix{
\triv\ar[r]&N\ar[r]& U\ar[r]& U_1\ar[r]&\triv}
\end{equation}
and evaluated on the discrete $\Z_p\dbl U\dbr$-module $\II_p(k)$, is concentrated on the first and the second row.
In particular, $d_r^{s,t}=0$ for $r\geq3$.
As \eqref{eq:spl4} splits, and as $\II_p(k)$ is inflated from $U_1$,
one has $E_2^{s,0}(\II_p(k))=E_\infty^{s,0}(\II_p(k))$ for $s\geq0$ (cf. \cite[Prop.~2.4.5]{nsw:cohn}).
Hence $d_2^{s,t}=0$ for all $s,t\geq0$, i.e., $E_2^{s,t}(\II_p(k))=E_\infty^{s,t}(\II_p(k))$, 
and the spectral sequence collapses.
Thus, using the isomorphism \eqref{eq:HI},
for every $k\geq1$ one has a short exact sequence 
\begin{equation}
\label{eq:ses HI split}
 \xymatrix@C=0.5truecm{0\ar[r] & H^k(U_1,\II_p(k))\ar[r]^-{\inf} & H^k(U,\II_p(k))\ar[r] & H^{k-1}(U_1,\II_p(k-1))\ar[r] & 0,}
\end{equation}
where the right- and left-hand side are divisible $p$-torsion modules. 
As such $\Z_p$-modules are injective, \eqref{eq:ses HI split} splits showing that
$H^k(U,\II_p(k))$ is $p$-divisible.
Therefore, by Proposition~\ref{prop:comp}, $(G,\theta)$ is cyclotomic.

Thus, by induction the claim holds for all split $\theta_2$-abelian groups $(G_2,\theta_2)$
satisfying $\rk(\Zen_{\theta_2}(G_2))<\infty$.
In general, as $Z$ is a torsion free abelian pro-$p$ group, there
exists an inverse system $(Z_i)_{i\in I}$ of closed subgroups of $Z$ such that
$Z/Z_i$ is torsion free, of finite rank, and $Z=\varprojlim_{i\in I} Z/Z_i$.
Since $Z_i$ is normal in $G$ and
$$(G/Z_i,\btheta)\simeq (G_1,\theta_1)\boxtimes (G_2/Z_i,\btheta_2)$$
is cyclotomically $p$-oriented, Corollary~\ref{cor:comp2}
yields the claim.
\end{proof}

The following theorem can be seen as a generalization of a result of A.~Wadsworth
 \cite[Thm.~3.6]{wadsworth}.

\begin{thm}
\label{thm:fibco}
Let $(G_i,\theta_i)$, $i=1,2$, be $p$-oriented profinite groups
satisfying $\image(\theta_1)=\image(\theta_2)$. Assume further that 
$(G_2,\theta_2)$ is split $\theta_2$-abelian. 
Then for $(G,\theta)=(G_1,\theta_1)\boxtimes(G_2,\theta_2)$ one has that
\begin{equation}
\label{eq:fibco1}
H^\bullet(G,\htheta)\simeq H^\bullet(G_1,\htheta_1)\otimes
\Lambda_\bullet\left((\kernel(\theta_2)/\kernel(\theta_2)^p)^\ast\right).
\end{equation}
Moreover, if $(G_1,\theta_1)$ is Bloch-Kato, then $(G,\theta)$ is Bloch-Kato.
\end{thm}

\begin{proof}
Assume first that $\mathrm{d}(\Zen_{\theta_2}(G_2))$ is finite.
If $\mathrm{d}(\Zen_{\theta_2}(G_2))=1$ then one obtains the isomorphism \eqref{eq:fibco1} from \cite[Thm.~3.1]{wadsworth},
which uses the Hochschild-Serre spectral sequence associated to the short exact sequence of profinite groups 
\[\xymatrix{
\triv\ar[r]&\Zen_{\theta_2}(G_2)\ar[r]& G\ar[r]& G/\Zen_{\theta_2}(G_2)\ar[r]&\triv}
\]
and evaluated on the discrete $\Z_p\dbl G\dbr$-module $\F_p(k)$, to compute $H^\bullet(G,\htheta)$.
If $\mathrm{d}(\Zen_{\theta_2}(G_2))>1$, then applying induction on $\mathrm{d}(\Zen_{\theta_2}(G_2))$ yields the isomorphism 
\eqref{eq:fibco1}.
Finally, if $\Zen_{\theta_2}(G_2)$ is not finitely generated, then a limit argument similar to the one used in the proof Theorem~\ref{thm:cycfib} and Corollary~\ref{cor:quadlim} yield the claim.
\end{proof}

\subsection{Coproducts}

\label{ss:coproducts}
For two profinite groups $G_1$ and $G_2$ let $G=G_1\amalg G_2$ denote the {\sl coproduct} (or free product)
in the category of profinite groups (cf. \cite[\S~9.1]{ribzal:book}).
In particular, if $(G_1,\theta_1)$ and $(G_2,\theta_2)$ are two $p$-oriented profinite groups, the $p$-orientations $\theta_1$
and $\theta_2$ induce a $p$-orientation $\theta\colon G\to \Z_p^\times$ via the universal property of 
of the free product.
Thus, we may interpret $\amalg$ as the coproduct in the category of $p$-oriented profinite groups
(cf. \cite[\S3]{ido:small}). The same applies to $\ppp$ --- the coproduct in the category of pro-$p$ groups.

\begin{thm}\label{thm:freeprod}
Let $(G_1,\theta_1)$ and $(G_2,\theta_2)$ be two cyclotomically $p$-oriented profinite groups.
Then their coproduct $(G,\theta)=(G_1,\theta_1)\amalg(G_2,\theta_2)$ is cyclotomically oriented.
Moreover, if $(G_1,\theta_1)$ and $(G_2,\theta_2)$ are Bloch-Kato, then $(G,\theta)$ is Bloch-Kato.
\end{thm}

\begin{proof}
Let $(U,\theta\vert_U)$ be an open subgroup of $(G,\theta)$.
Then, by the Kurosh subgroup theorem (cf. \cite[Thm.~9.1.9]{ribzal:book}),
\begin{equation}
 \label{eq:kurosh}
U\simeq\coprod_{s\in\mathcal{S}_1}({^s}G_1\cap U)\amalg
\coprod_{t\in\mathcal{S}_2}({^t}G_2\cap U)\amalg F,
\end{equation}
where ${^y}G_i=yG_iy^{-1}$ for $y\in G$.
The sets $\mathcal{S}_1$ and $\mathcal{S}_2$
are sets of representatives of the double cosets
$U\backslash G/G_1$ and $U\backslash G/G_2$, respectively.
In particular, the sets $\ca{S}_1$ and $\ca{S}_2$ are finite, and $F$ is a free profinite subgroup of finite rank.

Put $U_s={}^sG_1\cap U$ for all $s\in\ca{S}_1$, and $V_t={}^tG_2\cap U$ for all $t\in\ca{S}_2$.
By \cite[Thm.~4.1.4]{nsw:cohn}, one has an isomorphism
\begin{gather}
H^k(U,\II_p(k))\simeq\bigoplus_{s\in\mathcal{S}_1}H^k(U_s,\II_p(k))\oplus
\bigoplus_{t\in\mathcal{S}_2}H^k(V_t,\II_p(k)), \label{eq:freeprod1}\\
\intertext{for $k\geq 2$, and an exact sequence}
\xymatrix{  M \ar[r]^-{\alpha} & H^1(U,\II_p(1))\ar[r] & M'\ar[r] & 0.
} \label{eq:freeprod2}
\end{gather}
If $(G_1,\theta_1)$ and $(G_2,\theta_2)$ are cyclotomically $p$-oriented, 
then, by hypothesis and \eqref{eq:freeprod1}, $H^k(U,\II_p(k))$ is a divisible $p$-torsion module for $k\geq2$. 
In \eqref{eq:freeprod2}, the module $M$ is a homomorphic image of a $p$-divisible $p$-torsion module,
and the module $M'$ is the direct sum of $p$-divisible $p$-torsion modules,
showing that $H^1(U,\II_p(1))$ is divisible.
Hence, by Proposition~\ref{prop:comp} and Corollary~\ref{cor:comp2}, $(G,\theta)$ is cyclotomically $p$-oriented.

Assume that $(G_1,\theta_1)$ and $(G_2,\theta_2)$ are Bloch-Kato.
Then --- for $U$ as in \eqref{eq:kurosh} --- one has by \eqref{eq:freeprod1} and \eqref{eq:freeprod2} that
\begin{equation}
\label{eq:freeprod3}
H^\bullet(U,\htheta\vert_U)\simeq
\bA\oplus
\bigoplus_{s\in\ca{S}_1} H^\bullet(U_s,\htheta\vert_{U_s}) \oplus 
\bigoplus_{t\in\ca{S}_2} H^\bullet(V_t,\htheta\vert_{V_t}) \oplus 
H^\bullet(F,\htheta\vert_F)
\end{equation}
where $\bA$ is a quadratic algebra, and $\oplus$ denotes the {\sl direct sum} in the category of quadratic algebras
(cf. \cite[p.~55]{pp:quad}). In particular, $H^\bullet(U,\htheta\vert_U)$ is quadratic.
\end{proof} 

For pro-$p$ groups one has also the following.

\begin{thm}\label{thm:freeppp}
Let $(G_1,\theta_1)$ and $(G_2,\theta_2)$ be two cyclotomically oriented pro-$p$ groups.
Then their coproduct $(G,\theta)=(G_1,\theta_1)\ppp(G_2,\theta_2)$ is cyclotomically oriented.
Moreover, if $(G_1,\theta_1)$ and $(G_2,\theta_2)$ are Bloch-Kato, then $(G,\theta)$ is Bloch-Kato.
\end{thm}

\begin{proof}
The Kurosh subgroup theorem is also valid in the category of pro-$p$ groups
with $\ppp$ replacing $\amalg$ (cf. \cite[Thm.~9.1.9]{ribzal:book}),
and \eqref{eq:freeprod1} and \eqref{eq:freeprod2} hold also in this context (cf. \cite[Thm.~4.1.4]{nsw:cohn}).
Hence the proof for cyclotomicity can be transferred verbatim.
The Bloch-Kato property was already shown in \cite[Thm.~5.2]{cq:BK}.
\end{proof}

%%%%%%%%%%%%%%%%%%%%%%%%%%%%%%%%%%%%%%%%%%%%%%%%%%%%%%%%%%%%%%%%%%%%%%%%%%%%%%%%%%%%%%%%%%%%%%%%%%
%%%%%%%%%%%%%%
%%%%%%%%%%%%%%%%%%%%%%%%%%%%%%%%%%%%%%%%%%%%%%%%%%%%%%%%%%%%%%%%%%5%%%%%%%%%%%%%%%%%%%%%%%%%%%%%%%%

%%%%%%%%%%%%%%%%%%%%%%%%%%%%%%%%%%%%%%%%%%%%%%%%%%NEW
%%%%%  %%%%%%%%%%%%%%%
%%%%%%%%%%%%%%%%%%%%%%%%%%%%%%%%%%%%%

%%%%%%%%%%%%%%%%%%%%%%%%%%%%%%
%%%%%%%%%%% to be copied here %%%%%%%%%%
%%%%%%%%%%%%%%%%%%%%%%%%%%%%%%
\section{Oriented virtual pro-$p$ groups}\label{s:virtual}

We say that a $p$-oriented profinite group $(G,\theta)$ is an 
{\sl oriented virtual pro-$p$ group}
if $\kernel(\theta)$ is a pro-$p$ group.
In particular, $G$ is a virtual pro-$p$ group.
Since $\Z_2^\times$ is a pro-2 group,
every oriented virtual pro-2 group is in fact a pro-2 group.
For $p\neq 2$ let $\hat{\theta}\colon G\to{\F_p}^\times$ be the homomorphism induced by $\theta$, and put $G^\circ=\kernel(\hat{\theta})$.
Then $G/G^\circ\simeq\image(\hat{\theta})$ is a finite cyclic group of order co-prime to $p$.
The profinite version of the Schur-Zassenhaus theorem (cf. \cite[Lemma~22.10.1]{friedjarden:FA})
implies that the short exact sequence of profinite groups 
\begin{equation}\label{eq:ses virtual}
 \xymatrix@C=1.1truecm{ \{1\}\ar[r] & G^\circ\ar[r] & G\ar[r]^-{\hat{\theta}} & \image(\hat{\theta})\ar[r]\ar@/_2pc/[l]_-{\sigma} & \{1\} }
\end{equation}
splits. Indeed, if $C\subseteq G$ is a $p^\prime$-Hall subgroup of $G$, then 
$\pi\vert_C\colon C\to\image(\hat{\theta})$ is an isomorphism, and 
$\sigma=(\pi\vert_C)^{-1}$ is a canonical section for $\hat{\theta}$.

Note that $\Z_p^\times=\F_p^\times\times \Xi_p$, where $\Xi_p=O_p(\Z_p^\times)$
is the pro-$p$ Sylow subgroup of $\Z_p^\times$, and where we denoted by $\F_p^\times$ 
also the image of the Teichm\"uller section $\tau\colon\F_p^\times\to\Z_p^\times$. 
Hence a $p$-orientation $\theta\colon G\to\Z_p^\times$ on $G$ defines
a homomorphism $\hat{\theta}\colon G\to\F_p^\times$ and also a homomorphism
$\theta^\vee\colon G\to\Xi_p$. On the contrary a pair of continuous homomorphisms
$(\hat{\theta},\theta^\vee)$, where
$\hat{\theta}\colon G\to\F_p^\times$ and $\theta^\vee\colon G\to\Xi_p$, defines a
$p$-orientation $\theta\colon G\to \Z_p^\times$ given by
$\theta(g)=\hat{\theta}(g)\cdot\theta^\vee(g)$ for $g\in G$.

\begin{fact}
\label{fact:virtor}
Let $\hat{\theta}\colon G\to\F_p^\times$, $\sigma\colon \image(\hat{\theta})\to G$ be homomorphisms of groups satisfying \eqref{eq:ses virtual}. A homomorphism
$\theta^\circ\colon G^\circ\to\Xi_p$ defines a $p$-orientation $\theta\colon G\to\Z_p^\times$, provided for all 
$c\in\image(\hat{\theta})$ and for all $g\in G^\circ$ one has
\begin{equation}
\label{eq:conc1}
\theta^\circ(\sigma(c)\cdot g\cdot \sigma(c)^{-1})=\theta^\circ(g)
\end{equation}
\end{fact}

\begin{proof}
By \eqref{eq:ses virtual}, one has $G=G^\circ\rtimes_\beta\bar{\Sigma}$, where
$\bar{\Sigma}=\image(\hat{\theta})$, $\beta\colon\bar{\Sigma}\to\Aut(G^\circ)$ and $\beta(c)$ is left conjugation by $\sigma(c)$ for $c\in\bar{\Sigma}$.
Thus, by \eqref{eq:conc1}, the map $\theta^\vee\colon G\to\Xi_p$ given by
$\theta^\vee(g,c)=\theta^\circ(g)$ is a continuous homomorphism of groups, and 
$(\iota,\theta^\vee)$, where $\iota\colon\bar{\Sigma}\to\F_p^\times$ is the canonical inclusion,
defines a $p$-orientation of $G$.
\end{proof}

Let $(G,\theta)$ be an oriented virtual pro-$p$ group satisfying \eqref{eq:ses virtual}. 
As $\theta\colon G\to\Z_p^\times$ is a homomorphism onto an abelian group one has
\begin{equation}
\label{eq:conc2}
\theta(c\cdot g\cdot c^{-1})=\theta(g)
\end{equation}
for all $c\in C=\image(\sigma)$ and $g\in G$. Thus, if $i_c\in\Aut(G)$ denotes left conjugation by 
$c\in C$,
one has 
\begin{equation}
\label{eq:invcon}
\theta=\theta\circ i_c
\end{equation}
for all $c\in C$.

%%%%%%%%%%%%%%%%%%%%

\subsection{Oriented $\bar{\Sigma}$-virtual pro-$p$ groups}
\label{ss:Obsig}
From now on let $p$ be odd, and fix a subgroup $\bar\Sigma$ of $\F_p^\times$.
An  oriented virtual pro-$p$ group $(G,\theta)$ is said to be an
oriented $\bar{\Sigma}$-virtual pro-$p$ group, if $\image(\hat{\theta})=\bar{\Sigma}$.
Hence, by the previous subsection, for such a group one has a split short exact
sequence
\begin{equation}\label{eq:ses virtual2}
 \xymatrix{ \{1\}\ar[r] & G^\circ\ar[r] & G\ar[r]^{\hat{\theta}} & \bar\Sigma\ar[r]\ar@/_1pc/[l]_-{\sigma} & \{1\} }.
\end{equation}
By abuse of notation, we consider from now on $(G,\theta,\sigma)$ as
an oriented $\bar{\Sigma}$-virtual pro-$p$ group. As the following fact shows there is also
\emph{an alternative form} of a $\bSigma$-virtual pro-$p$ group.

\begin{fact}
\label{fact:virtpalt}
Let $\bSigma$ be a subgroup of $\F_p^\times$. Let $Q$ be a pro-$p$ group, let
$\theta^\circ\colon Q\to\Xi_p$ be a continuous homomorphism, and
let $\gamma_Q\colon\bSigma\to\Aut_c(Q)$ be a homomorphism of groups, where 
$\Aut_c(\argu)$ is the group of continuous automorphisms,
satisfying
\begin{equation}
\label{eq:pvirt1}
\begin{gathered}
\theta^\circ(\gamma_Q(c)(q))=\theta^\circ(q),\\
\end{gathered}
\end{equation}
for all $q\in Q$ and $c\in\bSigma$, then $(Q\rtimes_{\gamma_Q}\bSigma,\theta,\iota)$ is an
oriented $\bSigma$-virtual pro-$p$ group, where $\iota\colon\bSigma\to Q\rtimes_{\gamma_Q}\bSigma$ is the canonical map, and $\theta\colon Q\rtimes_{\gamma_Q}\bSigma\to\Z_p^\times$ is the homomorphism induced by 
$\theta^\circ$ (cf. Fact~\ref{fact:virtor}).
\end{fact}

If $(G_1,\theta_1,\sigma_1)$ and  $(G_2,\theta_2,\sigma_2)$
are oriented $\bar{\Sigma}$-virtual pro-$p$ groups, a continuous group homomorphism
$\phi\colon G_1\to G_2$ is said to be a morphism of  $\bar\Sigma$-virtual pro-$p$ groups,
if $\sigma_2=\phi\circ\sigma_1$ and $\theta_1=\theta_2\circ\phi$.
Similarly, if $(Q,\theta_Q^\circ,\gamma_Q)$ and  $(R,\theta_R^\circ,\gamma_R)$
are $\bSigma$-virtual pro-$p$ groups in alternative form (cf. Fact~\ref{fact:virtpalt}), the continuous group homomorphism $\phi\colon Q\to R$ is a homomorphims of $\bSigma$-virtual pro-$p$ groups provided
$\theta_R\circ\phi=\theta_Q$ and if for all $c\in\bSigma$ and for all $q\in Q$ one has that
\begin{equation}
\label{eq:pvirt2}
\gamma_R(c)(\phi(q))=\phi(\gamma_Q(c)(q)).
\end{equation}

With this slightly more sophisticated set-up the category of $\bar\Sigma$-virtual pro-$p$ groups admits coproducts. In more detail, let
$(Q,\theta_Q^\circ,\gamma_Q)$ and  $(R,\theta_R^\circ,\gamma_R)$ be
 $\bar{\Sigma}$-virtual pro-$p$ groups in alternative form. Put $X=Q\amalg^{\mathrm{p}} R$. Then for every element $c\in\bar{\Sigma}$ there exists an element $\delta(c)\in\Aut(X)$ making the
diagram
\begin{equation}
\label{eq:diacop}
\xymatrix{Q\ar[r]^-{\iota_1}\ar[d]_{\gamma_Q(c)}& X\ar[d]^{\delta(c)}& 
R\ar[l]_-{\iota_2}\ar[d]^{\gamma_R(c)}\\
Q\ar[r]^-{\iota_1}& X& R\ar[l]_-{\iota_2}
}
\end{equation}
commute.  Since $\Xi_p$ is a pro-$p$ group, there exists a continuous
group homomorphism $\theta^\circ\colon X\to\Xi_p$ making the lower two rows of the
diagram 
\begin{equation}
\label{eq:unior}
\xymatrix{Q\ar[r]^-{j_Q}\ar[d]_{\gamma_Q(c)}&X\ar[d]^-{\delta(c)}&
R\ar[l]_-{j_R}\ar[d]^{\gamma_R(c)}\\
Q\ar[r]^-{j_Q}\ar[dr]_-{\theta_Q^\circ}&X\ar[d]^-{\theta^\circ}&
R\ar[l]_-{j_R}\ar[dl]^-{\theta_R^\circ}\\
&\Xi_p&}
\end{equation} 
commute. Since $\theta^\circ_{Q/R}=\theta^\circ_{Q/R}\circ \gamma_{Q/R}(c)$ for all $c\in\bar{\Sigma}$,
one has $\theta^\circ=\theta^\circ\circ\delta(c)$ for all $c\in\bar{\Sigma}$. The commutativity
of the diagram \eqref{eq:unior} yields that the group homomorphisms
$j_Q\colon (Q,\theta^\circ_Q,\gamma_Q)\to (X,\theta^\circ,\delta)$ and
$j_R\colon(R,\theta^\circ_R,\gamma_R)\to (X,\theta^\circ,\delta)$ are  homomorphisms
of oriented $\bSigma$-virtual pro-$p$ groups in alternative form. Moreover, one has the following.

\begin{prop}
\label{prop:virtcop}
The oriented $\bSigma$-virtual pro-$p$ group $(X,\theta^\circ,\delta)$ together with the homomorphisms $j_Q\colon Q\to X$, and $j_R\colon R\to X$ is a  coproduct in the category of oriented $\bSigma$-virtual pro-$p$ groups.
\end{prop}

\begin{proof}
Let $(H,\theta_H,\gamma_H)$ be an oriented $\bSigma$-virtual pro-$p$ group
in alternative form,
and let $\phi_Q\colon Q\to H$ and $\phi_R\colon R\to H$ be homomorphisms of oriented
$\bSigma$-virtual pro-$p$ groups in alternative form. Then there exists a unique homomorphism of pro-$p$ groups $\phi\colon X\to H$ making the diagram
concentrated on the second and third row of
\begin{equation}
\label{eq:unior2}
\xymatrix{Q\ar[r]^-{j_Q}\ar[d]_{\gamma_Q(c)}&X\ar[d]^-{\delta(c)}&
R\ar[l]_-{j_R}\ar[d]^{\gamma_R(c)}\\
Q\ar@/_1pc/[2,1]_{\theta_Q^\circ}\ar[r]^-{j_Q}\ar[dr]_-{\phi_Q}&X\ar[d]^-{\phi}&
R\ar@/^1pc/[2,-1]^{\theta_R^\circ}\ar[l]_-{j_R}\ar[dl]^-{\phi_R}\\
&H\ar[d]^{\theta_H^\circ}&\\
&\Xi_p&}
\end{equation} 
commute.
Since $\phi_{Q/R}\circ\gamma_{Q/R}(c)=\gamma_H(c)\circ\phi_{Q/R}$ for all
$c\in\bSigma$, the uniqueness of $\phi$ implies that 
$\phi\circ\delta(c)=\gamma_H(c)\circ\phi$ for all $c\in\bSigma$. As $\phi_Q\colon Q\to H$
and $\phi_R\colon R\to H$ are homomorphisms of $\bSigma$-virtual pro-$p$ groups,
one has that $\theta_{Q/R}^\circ=\theta_H^\circ\circ\phi_{Q/R}$. This implies that
$(\theta_H^\circ\circ\phi)\circ j_{Q/R}=\theta_{Q/R}^\circ$, and from the construction of $\theta^\circ\colon X\to\Xi_p$ one concludes that $\theta^\circ=\theta_H^\circ\circ\phi$. This implies that $\phi$ is a homomorphism of oriented $\bSigma$-virtual pro-$p$ groups.
\end{proof}

\begin{ex}
For $p=3$ set $\bar\Sigma=\F_3^\times=\{1,s\}$.
Then the free product $(\Z_3^\times,\iid)\amalg^{\bar\Sigma}(\Z_3^\times,\iid)$ is isomorphic
to $F\rtimes\bar{\Sigma}$, where $F=\langle\,  x,y \,\rangle$ is a free pro-$3$ group of rank $2$
and the induced isomorphism $s\colon F\to F$ satisfies $s(x)=x^{-1}$, $s(y)=y^{-1}$.
\end{ex}

\begin{prop}
\label{prop:maxSig}
Let $(Q,\theta_Q,\gamma_Q)$ be an oriented $\bSigma$-virtual pro-$p$ group,
and let $\Zen$ be a normal $\bSigma$-invariant subgroup of $Q$
isomorphic to $\Z_p$, which is not contained in the Frattini subgroup
$\Phi(Q)=\cl([Q,Q]Q^p)$ of $Q$.
Then there exists a maximal closed subgroup $M$ of $Q$ which is
$\bSigma$-invariant, such that $M\cdot\Zen=Q$ and $M\cap\Zen=\Zen^p$.
\end{prop}

\begin{proof}
Let $\baQ=Q/\Phi(Q)$. Then $\gamma_Q$ induces a homomorphism
$\bgamma_{\baQ}\colon\bSigma\to\Aut_c(\baQ)$ making $\baQ$ a compact
$\F_p[\bSigma]$-module. Let $\Omega=\Hom_{\bSigma}^c(\baQ,\F_p)$, where $\F_p$ denotes the
finite field $\F_p$ with canonical left $\bSigma$-action.
By Pontryagin duality, one has $\bigcap_{\omega\in\Omega}\kernel(\omega)=\{0\}$. Thus, by hypothesis, there exists
$\psi\in\Omega$ such that $\psi\vert_\Zen\not=0$.
Hence $M=\kernel(\psi)$ has the desired properties.
\end{proof}

%%%%%%%%%%%%%%%%%%%%%%%%%%%%%%%%%%%%%

\subsection{The maximal oriented virtual pro-$p$ quotient}

For a prime $p$ and a profinite group $G$ we denote by $O^p(G)$  the closed subgroup of $G$
generated by all Sylow pro-$\ell$ subgroups of $G$, $\ell\not=p$.
In particular, $O^p(G)$ is {\sl $p$-perfect}, i.e., $H^1(O^p(G),\F_p)=0$,
and one has the short exact sequence 
\[ \xymatrix{\{1\}\ar[r] & O^p(G)\ar[r] & G\ar[r] & G(p) \ar[r] & \{1\}},\]
where $G(p)$ denotes the {\sl maximal pro-$p$ quotient} of $G$.

For a $p$-oriented profinite group $(G,\theta)$, we denote by
 \[G(\theta)=G/O^p(G^\circ)\]
 the {\sl maximal $p$-oriented virtual
pro-$p$ quotient} of $G$ (for the definition of $G^\circ$ see the beginning of \S~\ref{s:virtual}).
By construction, it carries naturally a $p$-orientation $\theta\colon G(\theta)\to\Z_p^\times$ inherited by $G$.

Note that if $\image(\theta)$ is a pro-$p$ group, then $G^\circ=G$, and $G(\theta)=G(p)$.

\begin{prop}
\label{propB}
 Let $(G,\theta)$ be a $p$-oriented Bloch-Kato profinite group,
 and let $O\subseteq G$ be a $p$-perfect subgroup such that $O\subseteq\kernel(\theta)$.
Then the inflation map
\begin{equation}
\label{eq:infGp}
\ifl^k(M)\colon H^k_{\cts}(G/O,M)\longrightarrow H^k_{\cts}(G,M),
\end{equation}
is an isomorphism for all $k\geq 0$ and all $M\in\ob({}_{\Z_p\dbl G/O\dbr}\prf)$, where
${}_{\Z_p\dbl G/O\dbr}\prf$ denotes the abelian category of profinite left 
$\Z_p\dbl G/O\dbr$-modules.
\end{prop}

\begin{proof}
As $O\subseteq\kernel(\theta)$, 
$\Z_p(k)$ is a trivial $\Z_p\dbl O\dbr$-module for every $k\in\Z$.
Since $O$ is $p$-perfect, and as the $\F_p$-algebra $H^\bullet(O,\F_p)$ is quadratic, 
$H^\bullet(O,\F_p)$ is 1-dimensional concentrated in degree 0.
By Pontryagin duality, this is equivalent to $H_k(O,\F_p)=0$ for all $k>0$,
where $H_k(O,\argu)$ denotes Galois homology as defined by A.~Brumer in \cite{brum:pseudo}.
Thus, the long exact sequence in Galois homology implies that
$H_k(O,\Z_p)=0$ for all $k>0$. 

Let $(P_\bullet,\der_\bullet,\eps)$ be a projective resolution of the trivial left
$\ZpG$-module in the category $\ZpGprf$. For a projective left $\ZpG$-module $P\in\ob(\ZpGprf)$ define
\begin{equation}
\label{eq:defdef}
\dfl(P)=\dfl^G_{G/O}(P)=\Z_p\dbl G/O\dbr\,\hotimes_G\, P,
\end{equation}
where $\hotimes$ denotes the completed tensor product as defined in \cite{brum:pseudo}.
Then, by the Eckmann-Shapiro lemma in homology, one has that
\begin{equation}
\label{eq:def1}
H_k(\dfl(P_\bullet),\dfl(\der_\bullet))\simeq H_k(O,\Z_p).
\end{equation}
Hence, by the previously mentioned remark, $(\dfl(P_\bullet),\dfl(\der_\bullet))$ is a 
projective resolution of $\Z_p$ in the category ${}_{\Z_p\dbl G/O\dbr}\mathbf{prf}$.

Let $M\in\ob({}_{\Z_p\dbl G/O\dbr}\mathbf{prf})$.
Then for every projective profinite left $\ZpG$-module $P$, one has a natural isomorphism
\begin{equation}
\label{eq:def2}
\Hom_{G/O}(\dfl(P),M)\simeq\Hom_G(P,M).
\end{equation}
Hence $\Hom_{G/O}(\dfl(P_\bullet),M)$ and $\Hom_G(P_\bullet,M)$ are 
isomorphic co-chain complexes, and the
induced maps in cohomology --- which coincide with $\ifl^\bullet(M)$ --- are isomorphisms.
\end{proof}

\begin{cor}\label{cor:hered_cyc}
  Let $(G,\theta)$ be a $p$-oriented profinite group which is Bloch-Kato, respectively cyclotomically oriented.
Then the maximal oriented virtual pro-$p$ quotient $(G(\theta),\theta)$ is Bloch-Kato,
 respectively cyclotomically oriented.
\end{cor}

%%%%%%%%%%%%%%%%%%%%%%%%%%%%%%%%%
%%%% Poincare duality, Thomas, 15/3/2016 %%%%%%%%%%%
%%%%%%%%%%%%%%%%%%%%%%%%%%%%%%%%%

%%%%%%%%%%%%%%%%%%%%%%%%%%%%%%%%%%%%%%%%%%%%%%%%%%NEW
%%%%%  %%%%%%%%%%%%%%%
%%%%%%%%%%%%%%%%%%%%%%%%%%%%%%%%%%%%%

%%%%%%%%%%%%%%%%%%%%%%%%%%%%%%%%%%5%%%%%%%%%%%%%%%%%%%%%%%%%
%%%%%%%5
%%%%%%%%%%%%%%%%%%%%%%%%%%%%%%%%%

%%%%%%%%%%%%%%%%%%%%%%%%%%%%%%%%%
%%%% Poincare duality, Thomas, 15/3/2016 %%%%%%%%%%%
%%%%%%%%%%%%%%%%%%%%%%%%%%%%%%%%%

\section{Profinite Poincar\'e duality groups and $p$-orientations}
\label{s:PDgps}
%%%%%%

\subsection{Profinite Poincar\'e duality groups}
\label{ss:PDpor}
Let $G$ be a profinite group, and let $p$ be a prime number.
Then $G$ is called a {\sl $p$-Poincar\'e duality group} of
dimension $d$, if
\begin{itemize}
\item[(PD$_1$)] $\ccd_p(G)=d$;
\item[(PD$_2$)] $|H^k_{\cts}(G,A)|<\infty$ for every finite discrete left $G$-module $A$ of $p$-power order;
\item[(PD$_3$)] $H^k_{\cts}(G,\ZpG)=0$ for $k\not= d$, and $H^d_{\cts}(G,\ZpG)\simeq\Z_p$.
\end{itemize}
Although quite different at first glance, for a pro-$p$ group our definition of $p$-Poincar\'e duality
coincides with the definition given by J-P.~Serre in \cite[\S I.4.5]{ser:gal}.
However, some authors prefer to omit the condition (PD$_2$) in the definition of a
$p$-Poincar\'e duality group (cf. \cite[Chap.~III, \S7, Definition~3.7.1]{nsw:cohn}).

For a profinite $p$-Poincar\'e duality group $G$ of dimension $d$
the profinite right $\ZpG$-module $D_G=H^d_{\cts}(G,\ZpG)$ is called the {\sl dualizing module}. 
Since $D_G$ is isomorphic to $\Z_p$ as a pro-$p$ group, there exists a unique $p$-orientation $\eth_G\colon G\to\Z_p^\times$
such that for $g\in G$ and $z\in D_G$ one has $$z\cdot g=z\cdot\eth_G(g)=\eth_G(g)\cdot z.$$
We call $\eth_G$ the {\sl dualizing $p$-orientation}.

Let ${}^\times D_G$ denote the associated 
profinite left $\ZpG$-module, i.e., setwise ${}^\times D_G$ coincides with $D_G$ and
for $g\in G$ and $z\in {}^\times D_G$ one has \[g\cdot z=z\cdot g^{-1}=\eth_G(g^{-1})\cdot z.\]
For a profinite $p$-Poincar\'e duality group of dimension $d$ the usual standard arguments 
(cf. \cite[\S VIII.10]{coh:brown} for the discrete case) provide
natural isomorphisms
\begin{equation}
\label{eq:natPD}
\begin{aligned}
\Tor^G_k(D_G,\argu)&\simeq H^{d-k}_{\cts}(G,\argu),\\
\Ext^k_G({}^\times D_G,\argu)&\simeq H_{d-k}(G,\argu),
\end{aligned}
\end{equation}
where $\Tor^G_\bullet(\argu,\argu)$ denotes the left derived functor of
$\argu\hotimes_G\argu$, and $\Ext^\bullet_G(\argu,\argu)$
denotes the right derived functors of $\Hom_G(\argu,\argu)$
in the category $\ZpGprf$ (cf. \cite{brum:pseudo}).

If $A$ is a discrete left $G$-module which is also a $p$-torsion module, then $A^\ast$
carries naturally the structure of a left (profinite) $\ZpG$-module (cf. \cite[p.~171]{ribzal:book}).
Then, by \cite[\S~I.3.5, Proposition~17]{ser:gal}, Pontryagin duality and \cite[(3.4.5)]{sw:cpg}, one obtains
for every finite discrete left $\ZpG$-module $A$ of $p$-power order that
\begin{equation}
\label{eq:natPD2}
H^d_{\cts}(G,A)\simeq \Hom_G(A,I_G)^\ast\simeq \Hom_G(I_G^\ast,A^\ast)^\ast
\simeq(I_G^\ast)^\times\,\hotimes_G\, A,
\end{equation}
where $I_G$ denotes the discrete left dualizing module of $G$ (cf. \cite[\S I.3.5]{ser:gal}).
In particular, by \eqref{eq:natPD}, $D_G\simeq (I^\ast_G)^\times$.

\begin{ex}
\label{ex:ladic}
Let $G_\K$ be the absolute Galois group of an $\ell$-adic field $\K$. Then 
$G_\K$ satisfies $p$-Poincar\'e duality of dimension $2$ for all prime numbers $p$.
One has $I_G\simeq\mu_{p^\infty}(\bar\K)$ (cf. \cite[\S II.5.2, Theorem~1]{ser:gal}). Hence 
${}^\times D_{G_\K}\simeq \Z_p(-1)$
with respect to the cyclotomic $p$-orientation $\theta_{\K,p}\colon G_{\K}\to\Z_p^\times$, i.e.,
 $\eth_{G_\K}=\theta_{\K,p}$.
\end{ex}

As we will see in the next proposition, the final conclusion in Example~\ref{ex:ladic} 
is a consequence of a general property of Poincar\'e duality groups.

\begin{prop}
\label{prop:woPD}
Let $G$ be a $p$-Poincar\'e duality group of dimension $d$, and let 
$\theta\colon G\to\Z_p^\times$ be a cyclotomic $p$-orientation of $G$. 
Then $\theta^{d-1}=\eth_G$ and  ${}^\times D_G\simeq \Z_p(1-d)$.
\end{prop}

\begin{proof}
By \eqref{eq:natPD} and the hypothesis, 
$H^d_{\cts}(G,\Z_p(d-1))\simeq D_G\, \hotimes\,\Z_p(d-1)$
is torsion free, and hence isomorphic to $\Z_p$.
This implies $\eth_G=\theta^{d-1}$.
\end{proof}

%%%%%%%%%%%%%%%%%%%%%%%%%%%%%%%%%%%

\subsection{Finitely generated $\theta$-abelian pro-$p$ groups}
\label{ss:thetaab}
Recall that $(G,\theta)$ is said to be $\theta$-abelian if $\kernel(\theta)=\Zen_\theta(G)$ and 
$\Zen_\theta(G)$ is $p$-torsion free --- in particular $\kernel(\theta)$ is an abelian pro-$p$ group.
If $G$ is finitely generated then one has an isomorphism of left $\ZpG$-modules $N\simeq \Z_p(1)^r$
for some non-negative integer $r$, and either $\Gamma=\image(\theta)$ is a finite group of order coprime to $p$, or
$\Gamma$ is a $p$-Poincar\'e duality group of dimension $1$ satisfying $\eth_\Gamma=\eqo_\Gamma$
(cf. \cite[Prop.~3.7.6]{nsw:cohn}).
Moreover, one has isomorphisms of left $\ZpG$-modules
\begin{equation}
\label{eq:tab1}
H_k(N,\Z_p)\simeq\Lambda_k(N)\simeq \Z_p(k)^{\binom{r}{k}},
\end{equation}
where $\Lambda_\bullet(\argu)$ denotes the exterior algebra over the ring $\Z_p$.
Since $\ccd_p(\Gamma)\leq 1$, the Hochschild-Serre spectral sequence for homology (cf. \cite[\S~6.8]{weibel:book})
\begin{equation}
\label{eq:HShom}
E_{s,t}^2=H_s(\Gamma,H_t(N,\Z_p(-m)))\ \Longrightarrow H_{s+t}(G,\Z_p(-m))
\end{equation}
is concentrated in the first two columns. Hence, the spectral sequence
collapses at the $E^2$-term, i.e., $E_{s,t}^2=E_{s,t}^\infty$. Thus, 
for $n\geq 1$ one has a short exact sequence
\begin{equation}
\label{eq:tab2}
\xymatrix@C=0.5truecm{
0\ar[r] &H_{n-1}(N,\Z_p(-m))^\Gamma\ar[r]& H_n(G,\Z_p(-m))\ar[r] &H_n(N,\Z_p(-m))_\Gamma\ar[r]&0
}
\end{equation}
if $\ccd_p(\Gamma)=1$, and isomorphisms
\begin{equation}
\label{eq:tab3}
 H^n(G,\Z_p(-m))\simeq H_n(N,\Z_p(-m))_\Gamma
\end{equation}
if $\Gamma$ is a finite group of order coprime $p$.
Here we used the fact that $H_0(\Gamma,\argu)=\argu_\Gamma$ coincides with the
coinvariants of $\Gamma$, and
that $H_1(\Gamma,\argu)=\argu^\Gamma$ coincides with the invariants of $\Gamma$
if $\Gamma$ is a $p$-Poincar\'e duality group of dimension $1$ with 
$\eth_\Gamma=\eqo_\Gamma$.
Since $H_{m-1}(N,\Z_p(-m))^\Gamma$ is a torsion free abelian pro-$p$ group, and as
\begin{equation}
\label{eq:tab4}
H_m(N,\Z_p(-m))_\Gamma=(H_m(N,\Z_p)\otimes\Z_p(-m))_\Gamma\simeq \Lambda_m(N)
\end{equation}
by \eqref{eq:tab1}, one concludes from \eqref{eq:tab2} and \eqref{eq:tab3} that $H_m(G,\Z_p(-m))$ is torsion free.

\begin{prop}
\label{prop:theta}
Let $(G,\theta)$ be a $\theta$-abelian $p$-oriented virtual pro-$p$ group such that
$N=\kernel(\theta)$ is a finitely generated torsion free abelian pro-$p$ group, and
that $\Gamma=\image(\theta)$ is $p$-torsion free.
Then $G$ is a $p$-Poincar\'e duality group of dimension $d=\ccd(G)$, and $\theta$ is cyclotomic.
\end{prop}

\begin{proof}
By hypothesis, $G$ is a $p$-torsion free $p$-adic analytic group. Hence the former assertion is a direct consequence of M.~Lazard's theorem (cf. \cite[Thm.~5.1.5]{sw:cpg}). The latter  follows from Proposition~\ref{prop:comp}.
\end{proof}

From Proposition~\ref{prop:woPD} one concludes the following:

\begin{cor}
\label{cor:tab}
Let $(G,\theta)$ be a $\theta$-abelian pro-$p$ group.
If $p=2$ assume further that $\image(\theta)$ is torsion free.
\begin{itemize}
\item[\textup{(}a\textup{)}] The orientation $\theta$ is cyclotomic.
\item[\textup{(}b\textup{)}] Suppose that $G$ is finitely generated with minimun number of generators $d=d(G)<\infty$.
If $p=2$ assume further that $\image(\theta)\subseteq 1+4\Z_2$.
Then $G$ is a Poincar\'e duality pro-$p$ group of dimension $d$. Moreover, 
$\eth_G=\theta^{d-1}$.
    \item[\textup{(}c\textup{)}] If $G$ satisfies the hypothesis of \textup{(}b\textup{)} and $d(G)\geq 2$, then for $p$ odd, any cyclotomic
orientation $\theta^\prime\colon G\to\Z_p^\times$ of $G$ must coincide with $\theta$, i.e., 
$\theta^\prime=\theta$. For $p=2$ any cyclotomic orientation $\theta^\prime\colon G\to\Z_2^\times$ satisfying 
$\image(\theta^\prime)\subseteq 1+4\Z_2$ must coincide with $\theta$.
\end{itemize}
\end{cor}

\begin{proof}
(a) follows from Proposition~\ref{prop:theta}.

\noindent
(b) By hypothesis, $G$ is uniformly powerful (cf. \cite[Ch.~4]{ddms:padic}), or
equi-$p$-value, as it is called in \cite{lazard:padic}. Hence the claim follows from 
Proposition~\ref{prop:theta}. By Proposition~\ref{prop:woPD}, $\eth_G=\theta^{d-1}$.

\noindent
(c) An element $\phi\in\Hom_{\grp}(G,\Z_p^\times)$ has finite order if, and only if,
$\image(\phi)$ is finite. Proposition~\ref{prop:woPD} and part (b) imply that
$$\theta^{d-1}=\eth_G=(\theta^{\prime})^{d-1}.$$ Hence $(\theta^{-1}\theta^\prime)^{d-1}=\eqo_G$.
For $p$ odd, $\Hom_{\grp}(G,\Z_p^\times)$ does not contain non-trivial elements of finite order.
Hence $\theta^\prime=\theta$. For $p=2$ the hypothesis implies that $\image(\theta^{-1}\theta^\prime)\subseteq 1+4\Z_2$. 
Hence $(\theta^{-1}\theta^\prime)^{d-1}=\eqo_G$ implies that $\theta^\prime=\theta$.
\end{proof}

Note that, by Fact~\ref{fact:ccd1}, Corollary~\ref{cor:tab}(c) cannot hold if $d(G)=1$.

%%%%%%%%%%%%5%%%%%%%%%%%5%%%%%%%%%%%

\subsection{Profinite $p$-Poincar\'e duality groups of dimension 2}
\label{ss:PDdim2} 
As the following theorem shows, 
for a profinite $p$-Poincar\'e duality group $G$ of dimension $2$, 
the dualizing $p$-orientation $\eth_G\colon G\to\Z_p^\times$ is always cyclotomic.

\begin{thm}
\label{thm:PD2}
Let $G$ be a profinite $p$-Poincar\'e duality group of dimension $2$. 
Then $\eth_G\colon G\to\Z_p^\times$ is a cyclotomic $p$-orientation.
\end{thm}

\begin{proof}
As every $p$-oriented profinite group is $0$-cyclotomic, 
it suffices to show that $H^2_{\cts}(U,\Z_p(1))$ is torsion free
for every open subgroup $U\subseteq G$. 
By Proposition~\ref{prop:woPD}, $\Z_p(-1)\simeq {}^\times D_G$.
Hence, from the Eckmann-Shapiro lemma in homology and \eqref{eq:natPD},
one concludes that
\begin{equation}
\label{eq:PDcyc2}
\begin{aligned}
H_1(U,\Z_p(-1))&=\Tor_1^U(\Z_p,\Z_p(-1))\simeq\Tor_1^U(\Z_p(-1)^\times,\Z_p)\\
&\simeq\Tor_1^G(D_G,\Z_p\dbl G/U\dbr)\simeq H^1_{\cts}(G,\Z_p\dbl G/U\dbr) \\ &\simeq \Hom_{\grp}(U,\Z_p).
\end{aligned}
\end{equation}
Hence $H_1(U,\Z_p(-1))$ is a torsion free $\Z_p$-module, and, by Proposition~\ref{prop:comp}, 
$H_{\cts}^2(U,\Z_p(1))$ is torsion free as well.
\end{proof}

\begin{rem}
Let $G$ be a profinite $p$-Poincar\'e duality group of dimension
$2$, and let $\eth_G\colon G\to\Z_p^\times$ be 
the dualizing $p$-orientation.
Then $(G,\eth_G)$ is not necessarily Bloch-Kato, as the following example shows.

Let $p=2$ and let $A=\PSL_2(q)$ where $q\equiv 3\mod 4$.
Then there exists a $p$-Frattini extension $\pi\colon G\to A$ of $A$
such that $G$ is a $2$-Poincar\'e duality group of dimension $2$,
i.e., $\kernel(\pi)$ is a pro-$2$ group contained in the Frattini subgroup of $G$
(cf. \cite{wei:stand}).
In particular, $G$ is perfect, and thus $\eth_G=\eqo_G$.
Hence $\F_2(1)=\F_2(0)$ is the trivial $\F_2\dbl G\dbr$-module, and
 --- as $G$ is perfect --- $H^1(G,\F_2(1))=0$.
Moreover, $H^2(G,\F_2(2))\simeq\F_2$, as $G$ is a profinite
$2$-Poincar\'e duality group of dimension 2 with $\eth_G=\eqo_G$.
Therefore, $H^\bullet(G,\eqo_G)$ is not quadratic.
\end{rem}

A pro-$p$ group $G$ which satisfies $p$-Poincar\'e duality in dimension $2$
is also called a {\sl Demu\v skin group} (cf. \cite[Def.~3.9.9]{nsw:cohn}).
For this class of groups one has the following.

\begin{cor}
\label{thm:propPD2}
Let $G$ be a Demu\v skin pro-$p$ group.
Then $G$ is a Bloch-Kato pro-$p$ group, and $\eth_G\colon G\to\Z_p^\times$ is a cyclotomic $p$-orientation.
\end{cor}

\begin{proof} By Theorem~\ref{thm:PD2}, it suffices to show that $(G,\eth_G)$ is Bloch-Kato.
It is well known that $H^\bullet(G,\hat\eth_G)$ is quadratic (cf. \cite[\S I.4.5]{ser:gal}).
Moreover, every open subgroup $U$ of $G$ is again a Demu\v skin group, with $\eth_U=\eth_G\vert_U$
(cf. \cite[Thm.~3.9.15]{nsw:cohn}).
Hence $(G,\eth_G)$ is Bloch-Kato.
\end{proof}

\begin{rem}
\label{rem:kb2}
[The Klein bottle pro-$2$ group]
Let $G$ be the pro-$2$ group given by the presentation
\begin{equation}
\label{eq:kb2}
G=\langle\, x,y\mid xyx^{-1}y=1\,\rangle
\end{equation}
Then $G$ is a Demu\v skin pro-$2$ group containing the free abelian pro-$2$ group
$H=\langle\,x^2,y \,\rangle$ of rank $2$. Thus, by Corollary~\ref{thm:propPD2} $(G,\eth_G)$ is cyclotomic. 
Since $H^1(G,\II_2(0))\simeq\II_2\oplus\Z/2\Z$, Proposition~\ref{prop:comp} implies that $\eth_G\not=\bone_G$ is non-trivial. In particular, since $\eth_G\vert_H=\bone_H$, this implies that $\image(\eth_G)=\{\,\pm1\,\}$. 
Note that $H=\kernel(\eth_G)$ and that one has a canonical isomorphism
\begin{equation}
\label{eq:isoKB}
H=\langle\,  x^2 \,\rangle\oplus\langle\,  y \,\rangle\simeq\Z_2(0)\oplus\Z_2(1).
\end{equation}
In particular, $(G,\eth_G)$ is not $\eth_G$-abelian.
\end{rem}

%\begin{rem}
%\label{rem:dem}
%The structure of Demu\v skin group has been analyzed in detail by J-P.~Serre (cf. \cite{ser:demu})
%and by J.~Labute (cf. \cite{labute:demushkin}).
%A Demu\v skin group is uniquely determined by the minimum number of generators and by the image of 
%$\eth_G$ (cf. \cite{labute:demushkin}, \cite[\S I.4.5, Exercise~2]{ser:gal}).
%\end{rem}
\begin{ex}
 Let $G$ be the pro-$p$ group with presentation
\[
 G=\langle\,  x,y,z\mid [x,y]=z^{-p} \,\rangle.
\]
If $p=2$ then $G$ is a Demu\v skin group, and $\eth_G\colon G\to\Z_2^\times$ is given by $\eth_G(x)=\eth_G(y)=1$,
$\eth_G(z)=-1$.
On the other hand, if $p\neq2$ then $G$ is not a Demu\v skin group, and any $p$-orientations
$\theta\colon G\to\Z_p^\times$ is not 1-cyclotomic (cf. \cite[Thm.~8.1]{eq:kummer}).
However, $H^\bullet(G,\hat \theta)$
is still quadratic.
\end{ex}

%%%%%%%%%%%%%%%%%%%%%%%%%%%%%%%%%
%%%%%%%%%%%%%%%%%%%%%%%%%%%%%%%%%%

%%%%%%%%%%%%%%%%%%%%%%%%%%%%%%%%%%

%%%%%%%%%%%%%%%%%%%%%%%%%%%%%%%%%%%%%%%%%%%%%%%%%%%%%%%%%%%%%%%%%5%%%%%%%%%%%%%%%%%%%%%%%%%%%
%%%%%%%%%%%%%%%%%%%%%%%%%55
%%%%%%%%%%%%%%%%%%%%%%%%%%%%%%%%%%%%%%%%%%%%%%%%%%%%%%%%%%%%%%%%%%%%%%%%%%%%%%%%%%%%%%%%%%

\section{Torsion}
\label{s:torBK}

It is well known that a Bloch-Kato pro-$p$ group may have non-trivial torsion only if, $p=2$.
More precisely, a Bloch-Kato pro-$2$ group $G$ is torsion if, and only if, $G$ is abelian and of exponent $2$.
Moreover, any such group is a Bloch-Kato pro-$2$ group 
(cf. \cite[\S2]{cq:BK}).
The following result --- which appeared first in \cite[Prop.~2.13]{cq:phd} --- holds for 1-cyclotomically oriented pro-$p$ groups
(see also \cite[Ex.~3.5]{eq:kummer} and \cite[Ex.~14.27]{smooth}).

\begin{prop}\label{prop:tor2}
Let $(G,\theta)$ be a 1-cyclotomically oriented pro-$p$ group.
\begin{itemize}
\item[(a)] If $\image(\theta)$ is torsion free, then $G$ is torsion free.
\item[(b)] If $G$ is non-trivial and torsion, then $p=2$, $G\simeq C_2$ and $\theta$ is injective.
\end{itemize}
\end{prop}

\begin{rem}
\label{rem:C2} Let $\theta\colon C_2\to\Z_2^\times$ be an injective homomorphism of groups. Then 
$\Z_2(1)\simeq\omega_{C_2}$ is isomorphic to the augmentation ideal
$$\omega_{C_2}=\kernel(\Z_2[C_2]\to\Z_2).$$  Hence --- by dimension shifting ---
$$H^2(C_2,\Z_2(1))=H^1(C_2,\Z_2(0))=0.$$ Thus --- as $C_2$ has periodic cohomology of period $2$ --- one concludes that $H^s(C_2,\Z_2(t))=0$ for $s$ odd and $t$ even, and also for 
$s$ even and $t$ odd. Hence $(C_2,\theta)$ is cyclotomic.
\end{rem}

%\begin{proof} 
%(a) Suppose that $G$ has non-trivial torsion, and that $\image(\theta)$ is torsion free.
%This would imply that
%$(C_p,\eqo_{C_p})$ is a cyclotomically oriented finite $p$-group.
%But $H^2(C_p,\Z_p)\simeq\Z/p\Z$, and thus $(C_p,\eqo_{C_p})$ is not cyclotomically oriented,
%a contradiction. Part (b) is a direct consequence of (a).
%\end{proof}

From Proposition~\ref{prop:tor2} and the profinite version of Sylow's theorem one concludes the following corollary,
which can be seen as a version of the Artin-Schreier theorem
for 1-cyclotomically $p$-oriented profinite groups.

\begin{cor}\label{cor:torsion}
Let $p$ be a prime number, and let $(G,\theta)$ be a profinite group with
a 1-cyclotomic $p$-orientation.
\begin{itemize}
\item[\textup{(}a\textup{)}] If $p$ is odd, then $G$ has no $p$-torsion.
\item[\textup{(}b\textup{)}] If $p=2$, then every non-trivial $2$-torsion subgroup is isomorphic
to $C_2$. Moreover, if $\image(\theta)$ has no 
$2$-torsion, then $G$ has no $2$-torsion. 
\end{itemize}
\end{cor}

\begin{rem}\label{rem:2tor}
Let $\theta\colon \Z_2\to\Z_2^\times$ be the homomorphism of groups
given by $\theta(1+\lambda)=-1$ and $\theta(\lambda)=1$
for all $\lambda\in 2\Z_2$. Then $\theta$ is a $2$-orientation of $G=\Z_2$ satisfying
$\image(\theta)=\{\pm1\}$. As $\ccd_2(\Z_2)=1$, Fact~\ref{fact:ccd1} implies that $(\Z_2,\theta)$ is Bloch-Kato and cyclotomically
2-oriented.
However, $\image(\theta)$ is not torsion free.
\end{rem}

%%%%%%%%%%%%%%%%%%%%%%%%%%%%%%%%%%%%%%%%%%%%%%%%%%%%%%%%%%%%%%%%%%%%%%%%%%%%%

\subsection{Orientations on $C_2\times\Z_2$}
\label{ss:orC2}
As we have seen in Proposition~\ref{prop:theta}, for $p$ odd, every
$\theta$-abelian oriented pro-$p$ group is cyclotomically $p$-oriented.
For $p=2$, this is not true. Indeed, one has the following.

\begin{prop}
\label{fact:C2Z2}
Any 2-orientation $\theta\colon G\to\Z_2^\times$ on $G\simeq C_2\times \Z_2$ is not 1-cyclotomic.
\end{prop} 

\begin{proof}
Suppose that $(G,\theta)$ is 1-cyclotomically 2-oriented.
Let $x,y$ be elements of $G$ such that $x^2=1$ and $\mathrm{ord}(y)=2^\infty$, and that $x,y$ generate $G$.
Proposition~\ref{prop:tor2} applied to the cyclic pro-2 group generated by $x$ yields $\theta(x)=-1$.
Put $\theta(y)=1+2\lambda$ for some $\lambda\in\Z_2$.
By \cite[Prop.~6]{labute:demushkin}, if $\theta$ is 1-cyclotomic then for any pair of elements $c_x,c_y\in\Z_2(1)$ there exists
a continuous crossed-homomorphism $c\colon G\to\Z_2(1)$ (i.e., a map satisfying $c(g_1g_2)=c(g_1)+\theta(g_1)c(g_2)$,
cf. \cite[p.~15]{nsw:cohn}) such that $c(x)=c_x$, $c(y)=c_y$.
Set $c_x=c_y=1$.
Then one computes
\[\begin{split}
   & c(xy)=c_x+\theta(x)c_y=1-1=0,\qquad\text{and}\\
   & c(yx)=c_y+\theta(y)c_x=1+1+2\lambda,
  \end{split}\]
which yields $\lambda=-1$.
The element $xy$ has the same properties as $y$. Hence the previously mentioned argument applied to the element $xy$ yields $\theta(xy)=1-2=-1$,
whereas $\theta(xy)=\theta(x)\theta(y)=1$, a contradiction.
\end{proof}

\begin{rem}\label{rem:centralizing}
From Proposition~\ref{prop:tor2} and Proposition~\ref{fact:C2Z2} one deduces that in a 1-cyclotomically $2$-oriented
pro-$2$ group, every element of order 2 is self-centralizing, which is a remarkable property of absolute Galois groups
(cf. \cite[Prop.~2.3]{formrealfields} and \cite[Cor.~2.3]{addistruct}).
\end{rem}

\begin{prop}
\label{prop:orC2}
Let $(G,\theta)$ be a $\theta$-abelian oriented pro-$2$ group.
Then $\theta$ is cyclotomic if, and only if, either
\begin{itemize}
 \item[\textup{(}a\textup{)}] $\image(\theta)$ is torsion free; or 
     \item[\textup{(}b\textup{)}] $\image(\theta)$ has order $2$.
\end{itemize}
In both these cases $(G,\theta)$ is split $\theta$-abelian.
\end{prop}

\begin{proof}
Assume first that $\image(\theta)$ is torsion free.
Then the short exact sequence $\{1\}\to\kernel(\theta)\to G\to\image(\theta)\to\{1\}$ splits,
as $\image(\theta)\simeq\Z_2$ is a projective pro-2 group.
Moreover, $(G,\theta)$ is cyclotomic by Proposition~\ref{prop:theta}.

Second assume that $\theta$ is cyclotomic, $p=2$ and that 
$\image(\theta)\supseteq\{\pm1\}$.
If $g\in G$ satisfies $\theta(g)=-1$, 
then $g^2\in\kernel(\theta)=\Zen_\theta(G)$, and consequently 
\[
 g^2=g\cdot g^2\cdot g^{-1}=(g^2)^{\theta(g)}=g^{-2},
\]
i.e., $g^4=1$.
Since $(\kernel(\theta),\mathbf{1})$ is cyclotomically 2-oriented, $\kernel(\theta)$ is torsion free, and one deduces that $g^2=1$.
Therefore, the short exact sequence 
\[\xymatrix{\{1\}\ar[r] & H\ar[r] & G\ar[r] & C_2\ar[r] &\{1\}}\]
splits (here $H=\kernel(\pi\circ\theta)$, where $\pi$ is the canonical epimorphism $\Z_2^\times\twoheadrightarrow\{\pm1\}$).
Since $(H,\theta\vert_H)$ is again cyclotomically 2-oriented and as
$\image(\theta\vert_H)$ is torsion free, $(H,\theta\vert_H)$ is split $\theta\vert_H$-abelian by the previously mentioned argument.
We claim that $H=\kernel(\theta)$.
Indeed, suppose there exists $h\in H$ such that $\theta(h)\neq1$.
Put $\lambda=(1+\theta(h))/2$ and  let $z=ghgh^{-1}=[g,h^{-1}]\in \kernel(\theta)$.
Then --- as $g=g^{-1}$ and $\theta(g)=-1$ --- one has 
\[ \begin{split}
  g(z^{\lambda}h^2)g^{-1} &= (gzg)^{\lambda}\cdot gh^2g \\ &=z^{-\lambda}\cdot(ghg)^{2}
= z^{-\lambda}\cdot(ghgh^{-1}\cdot h)^2 \\ &= z^{-\lambda}\cdot(zhzh^{-1}\cdot h^2)=z^{-\lambda+1+\theta(h)}h^2 \\ &=z^{\lambda}h^2,
 \end{split}\]
i.e., $g$ and $z^\lambda h^2$ commute which implies that 
$\langle\,  g, z^\lambda h^2 \,\rangle\simeq C_2\times\Z_p$
contradicting Proposition~\ref{fact:C2Z2}.
Therefore, $H=\kernel(\theta)$ is a free abelian pro-2 group, and $G\simeq H\rtimes C_2$.

Finally, let $p=2$ and assume that $\image(\theta)=\{\pm1\}$. By Remark~\ref{rem:C2},
we may also assume that $\kernel(\theta)$ is non-trivial. Then, either

\noindent
Case I: $\theta^{-1}(\{-1\})$ contains an element of order $2$ and $(G,\theta)$ is split $\theta$-abelian, i.e., $G\simeq \kernel(\theta)\rtimes C_2$ with $\kernel(\theta)$ a free abelian pro-$2$ group, or

\noindent
Case II: all elements in $x\in\theta^{-1}(\{-1\})$ are of infinite order. Then for $y\in\kernel(\theta)$, the group $K=\langle\,  x,y \,\rangle$ must be isomorphic to the
Klein bottle pro-$2$ group which is impossible as $G$ is $\theta$-abelian and thus contains
only $\theta$-abelian closed subgroups (cf. Remark~\ref{rem:thetab}(b)). Hence Case II is
impossible.

By Lemma~\ref{lem:splGr}, if $U\subseteq G$ is an open subgroup, then either $U\subseteq\kernel(\theta)$,
or $U\simeq V\rtimes C_2$ for some open subgroup $V$ of $\kernel(\theta)$.
In the first case, $(U,\mathbf{1})$ is cyclotomically 2-oriented by Proposition~\ref{prop:theta}.
For the second case, we claim that $H^k(U,\II_2(k))$ is 2-divisibe for all $k\geq1$.

Recall that $\Z_2[C_2]$ has periodic cohomology (of period $2$), and that one has the equalities of $\Z_2\dbl U\dbr$-modules
$\II_2(k)=\II_2(0)$ for $k$ even and $\II_2(k)=\II_2(-1)$ for $k$ odd.
Moreover,  
\begin{equation}\label{eq:tatecoh}
 \begin{split}
  &\hat H^0(C_2,\II_2(0))= \II_2(0)^{C_2}/N_{C_2}\II_2(0)=\II_2(0)/2\cdot \II_2(0)=0,\\
&\hat H^{-1}(C_2,\II_2(-1))= \kernel(N_{C_2})/\omega_{C_2}\II_2(-1)=\II_2(-1)/2\cdot\II_2(-1)=0,
 \end{split} \end{equation}
where $\hat H^k$ denotes Tate cohomology,
$N_{C_2}=\sum_{x\in C_2} x\in\Z_2[C_2]$ is the norm element, and $\omega_{C_2}$ is the augmentation ideal of the group algebra $\Z_2[C_2]$ (cf. \cite[\S~I.2]{nsw:cohn}).
Thus, by \eqref{eq:tatecoh}, one has 
\begin{equation}
\label{eq:tate3b}
H^m(C_2,\II_2(m))=\hat{H}^m(C_2,\II_2(m))\simeq \hat{H}^k(C_2,\II_2(k))=0,
\end{equation}
for all positive integers $m>0$ and $m\equiv k(\mod 2)$.

Suppose first that $V\simeq\Z_2$.
As in the proof of Theorem~\ref{thm:cycfib}, the $E_2$-term of the Hochschild-Serre spectral sequence associated                                                                                                                                                                                                                                                                                                                                                                        
to the short exact sequence $\{1\}\to V\to U\to C_2\to\{1\}$ evaluated on $\II_2(k)$ is concentrated in the first and the second row.
In particular, $d_2^{\bullet,\bullet}=0$ and thus $E_2^{s,t}(\II_2(k))=E_\infty^{s,t}(\II_2(k))$. 
Thus, by Fact~\ref{fact:HI}, for every $k\geq1$ one has a short exact sequence
\[
 \xymatrix@C=0.5truecm{ 0\ar[r] & H^k(C_2,\II_2(k))\ar[r] & H^k(U,\II_2(k))\ar[r] & H^{k-1}(C_2,\II_2(k-1))\ar[r] &0},
\]
and $H^k(C_2,\II_2(k))=0$ by \eqref{eq:tate2}.
Hence, $(U,\theta\vert_U)$ is cyclotomically 2-oriented by Proposition~\ref{prop:comp}.
If $V\simeq\Z_2^n$ with $n>1$, then $H^k(U,\II_2(k))=0$ by induction on $n$ and the previously mentioned argument.
Finally, Corollary~\ref{cor:comp2} yields the claim in case $V$ not finitely generated.
\end{proof}

%%%%%%%%%%%%%%%%%%%%%%%%%%%%%%%%%%
%%%%% Cyc. oriented pro-p, Thomas, 9/3/2016 %%%%%%%%%
%%%%%%%%%%%%%%%%%%%%%%%%%%%%%%%%%%

%%%%%%%%%%%%%%%%%%%%%%%%%%%%%%%%%%
%%%%% Cyc. oriented pro-p, Thomas, 9/3/2016 %%%%%%%%%
%%%%%%%%%%%%%%%%%%%%%%%%%%%%%%%%%%

\section{Cyclotomically oriented pro-$p$ groups}
\label{s:prop}

For a cyclotomically oriented pro-$2$ group $(G,\theta)$ satisfying
$\image(\theta)\subseteq 1+4\Z_2$ one has the following.

\begin{fact}\label{prop:bock}
Let $(G,\theta)$ be a pro-$2$ group with a cyclotomic orientation satisfying $\image(\theta)\subseteq1+4\Z_2$.
Then $\chi\cup\chi=0$ for all $\chi\in H^1(G,\F_2)$, i.e.,
the first Bockstein morphism $\beta^1\colon H^1(G,\F_2)\to H^2(G,\F_2)$ vanishes.
\end{fact}

\begin{proof}
Since $\image(\theta)\subseteq1+4\Z_2$, the action of $G$ on $\F_2(1)$ is trivial.
The epimorphism of $\Z_2\dbl G\dbr$-modules $\Z_2(1)/4\to\F_2$ induces a long exact sequence 
\begin{equation}
\label{eq:lesBock} 
\xymatrix@C=0.8truecm{ \cdots\ar[r]^-{2\cdot} & H^1(G,\Z_2(1)/4) \ar[r]^-{\pi_{2,1}^1} & H^k(G,\F_2)\ar`r[d]`[l]^{\beta^1} `[dll] `[dl] [dl]   \\
& H^{2}(G,\F_2)\ar[r]^-{2\cdot} & H^{2}(G,\Z_2(1)/4)\ar[r] & \cdots}
\end{equation}
where the connecting homomorphism is the first Bockstein morphism.
Since $\theta$ is cyclotomic, the map $\pi_{2,1}^1$ is surjective, and thus $\beta^1$ is the 0-map.
\end{proof}

%\begin{proof}
%By hypothesis, $C_4=\Z_2(1)/4\Z_2(1)$ is a trivial $\Z_2\dbl G\dbr$-module
%which is isomorphic --- as abelian group --- to the cyclic group of order $4$.
%Since $H^2_{\cts}(G,\Z_2(1))$ is torsion free, one has a commutative diagram with exact rows
%\begin{equation}\label{eq:tor1}
%\xymatrix{H^1_{\cts}\left(G,\Z_2(1)\right)\ar[r]^{4.}\ar[d]_{2.}&
%H^1_{\cts}\left(G,\Z_2(1)\right)\ar[r]\ar@{=}[d]&H^1\left(G,C_4\right)\ar[r]\ar[d]^{\pi}&0\\
%H^1_{\cts}\left(G,\Z_2(1)\right)\ar[r]^{2.}&H^1_{\cts}\left(G,\Z_2(1)\right)\ar[r]&H^1\left(G,\F_2\right)\ar[r]
%\ar[d]^{\beta^1}&0\\
%%&&H^2(G,\F_2)&}
%\end{equation}
%with $\pi$ the canonical map.
%Hence, by the weak four lemma (cf. \cite[Lemma~3.2]{mcl:hom}) applied to the diagram 
%\eqref{eq:tor1} extended by another column of $0$'s to the right, $\pi$ is surjective. 
%Therefore, 
% $\beta^1\colon H^1(G,\F_2)\to H^2(G,\F_2)$
%is the trivial map. 
%\end{proof}

\begin{rem}
\label{rem:fccd}
As before for a finitely generated pro-$p$ group $G$ let $d(G)$ denote its minimum number of generators.
If $p$ is odd and $G$ is a finitely generated Bloch-Kato pro-$p$ group, the cohomology ring
$(H^\bullet(G,\F_p),\cup)$ is a quotient of the exterior $\F_p$-algebra 
$\Lambda_\bullet=\Lambda_\bullet(H^1(G,\F_p))$.
In particular, $\ccd_p(G)\leq d(G)$.
Moreover, $\Lambda_{d(G)}$ is the unique minimal ideal of $\Lambda_\bullet$.
Hence equality of $\ccd_p(G)$ and $d(G)$ is equivalent to $H^\bullet(G,\F_p)$ being isomorphic to $\Lambda_\bullet$.
It is well known that this implies that $G$ is uniformly powerful (cf. \cite[Thm.~5.1.6]{sw:cpg}),
and that there exists a $p$-orientation $\theta\colon G\to\Z_p^\times$ such that $G$ is $\theta$-abelian
(cf. \cite[Thm.~4.6]{cq:BK}).

Let $p=2$, and let $(G,\theta)$ be a cyclotomically oriented Bloch-Kato pro-$2$ group satisfying
$\image(\theta)\subseteq 1+4\Z_2$. Then Proposition~\ref{prop:bock} implies that
the cohomology ring
$(H^\bullet(G,\F_2),\cup)$ is a quotient of the exterior $\F_2$-algebra 
$\Lambda_\bullet=\Lambda_\bullet(H^1(G,\F_2))$, and hence $\ccd_2(G)\leq d(G)$.
If $\ccd_2(G)=d(G)$, the previously mentioned argument,
Proposition~\ref{prop:bock} and
\cite{wei:2deg} imply that $G$ is uniformly powerful.
Finally, \cite[Thm.~4.11]{cq:BK} yields that $G$ is $\theta^\prime$-abelian 
for some orientation $\theta^\prime\colon G\to\Z_2^\times$.
Thus, if $d(G)\geq 2$, one has $\theta=\theta^\prime$ by Corollary~\ref{cor:tab}(c).
\end{rem}

From the above remark and J-P.~Serre's theorem (cf. \cite{ser:cd})
one concludes the following fact.

\begin{fact}
\label{fact:torf}
Let $(G,\theta)$ be a finitely generated cyclotomically oriented 
torsion free Bloch-Kato pro-$2$ group. Then $\ccd_2(G)<\infty$.
\end{fact}

%%%%%%%%%%%%%%%%%%%%%%%%%

\subsection{Tits' alternative}
\label{ss:2gen}
From Remark~\ref{rem:fccd} one concludes the following.

\begin{fact}
\label{fact:tits}
{\rm (a)}
Let $p$ be odd, and let $G$ be a Bloch-Kato pro-$p$ group satisfying $d(G)\leq 2$.
Then $G$ is either isomorphic to a free pro-$p$ group,
or $G$ is $\theta$-abelian for some orientation $\theta\colon G\to\Z_p^\times$.

\noindent
{\rm (b)} Let $p=2$, and let $(G,\theta)$ be a cyclotomically oriented 
Bloch-Kato pro-$2$ group satisfying $\image(\theta)\subseteq 1+4\Z_2$ and $d(G)\leq 2$.
Then $G$ is either isomorphic to a free pro-$2$ group,
or $G$ is $\theta$-abelian.
\end{fact}

In \cite[Thm.~4.6]{cq:BK} it was shown, that for $p$ odd any Bloch-Kato pro-$p$ group satisfies
a strong form of Tits' alternative (cf. \cite{tits:alt}), i.e., either $G$ contains a closed non-abelian free pro-$p$ 
subgroup, or there exists a $p$-orientation $\theta\colon G\to\Z_p^\times$ such that $G$ is $\theta$-abelian.
Using the results from the previous subsection and \cite[Thm.~4.11]{cq:BK},
one obtains the following version of
Tits' alternative if $p$ is equal to $2$. 

\begin{prop}\label{prop:pro-2 tits}
Let $(G,\theta)$ be a cyclotomically oriented virtual pro-2 group which is also Bloch-Kato,
such that $\image(\theta)\subseteq 1+4\Z_2$.
Then either $G$ contains a closed non-abelian free pro-$2$ subgroup; or
$G$ is $\theta$-abelian.
\end{prop}

\begin{proof}
As $\image(\theta)\subseteq 1+4\Z_2$,
Proposition~\ref{prop:tor2}-(a) implies that $G$ is torsion free.
From Proposition~\ref{prop:bock} one concludes that the first Bockstein morphism $\beta^1$ vanishes.
Thus, the hypothesis of \cite[Thm.~4.11]{cq:BK} are satisfied (cf. Remark~\ref{rem:fccd}), and this yields the claim.
\end{proof}

\begin{rem}
\label{rem:p2X}
Note that Proposition~\ref{prop:pro-2 tits} without the hypothesis $\image(\theta)\subseteq 1+4\Z_2$
does not remain true
(cf. Remark~\ref{rem:kb2}).
\end{rem}

%%%%%%%%%%%%%%%%%%%%

\subsection{The $\theta$-center}
\label{ss:thetaZ}
One has the following characterization of the $\theta$-center for a cyclotomically oriented
Bloch-Kato pro-$p$ group $(G,\theta)$.

\begin{thm}\label{prop:Zmxlabeliannormal}
Let $(G,\theta)$ be a cyclotomically oriented torsion free Bloch-Kato pro-$p$ group.
If $p=2$ assume further that $\image(\theta)\subseteq1+4\Z_2$. Then
$\Zen_\theta(G)$ is the unique maximal closed abelian normal subgroup of $G$ contained in $\kernel(\theta)$.
\end{thm}

\begin{proof}
Let $A\subseteq\kernel(\theta)$ be a closed abelian normal subgroup of $G$,
let $z\in A$, $z\not=1$, and let $x\in G$ be an arbitrary element.
Put $C=\cl(\langle\,  x,z \,\rangle)\subseteq G$.
Then either $C\simeq \Z_p$ or $C$ is a 2-generated pro-$p$ group.
Thus, by Fact~\ref{fact:tits}, one has to distinguish three cases:
\begin{itemize}
\item[(i)] $d(C)=1$;
\item[(ii)] $d(C)=2$ and $C$ is isomorphic to a free pro-$p$ group; or
\item[(iii)] $d(C)=2$ and $C$ is $\theta^\prime$-abelian for some $p$-orientation 
$\theta^\prime\colon C\to\Z_p^\times$.
\end{itemize}
In case (i), $x$ and $z$ commute.
If $C$ is generated by $z$, then $C\subseteq\kernel(\theta)$ and $\theta(x)=1$.
If $C$ is generated by $x$, then $z=x^\lambda$ for some $\lambda\in\Z_p$, and $1=\theta(z)=\theta(x)^\lambda$.
Hence $\theta(x)=1$, as $\image(\theta)$ is torsion free.
In both cases $$ xzx^{-1}=z=z^{\theta(x)}.$$

Case (ii) cannot hold: by hypothesis, $A\cap C\not=\triv$,
but free pro-$p$ groups of rank $2$ do not contain non-trivial closed abelian normal subgroups.

Suppose that case (iii) holds. Then $\theta^\prime=\theta\vert_C$ by Corollary~\ref{cor:tab}(c),
and $z\in\kernel(\theta\vert_C)=\Zen_{\theta\vert_C}(C)$.
Therefore, $$xzx^{-1}=z^{\theta\vert_C(x)}=z^{\theta(x)}.$$

Hence we have shown that
for all $z\in A$ and all $x\in G$ one has that $xzx^{-1}=z^{\theta(x)}$. This yields the claim.
\end{proof}

The above result can be seen as the group theoretic generalization of \cite[Corollary~3.3]{ek98:abeliansbgps}
and \cite[Thm.~4.6]{en94:pro2gps}. Note that in the case $p=2$
the additional hypothesis in Theorem~\ref{prop:Zmxlabeliannormal} 
is necessary (cf.~Remark~\ref{rem:kb2}). Indeed, if $G$ is the Klein bottle pro-$2$ group
then $\langle\,  x^2 \,\rangle$ is another maximal closed abelian normal subgroup of $G$
contained in $\kernel(\eth_G)$.

\begin{rem}\label{rem:arithmetic_Z}
 Let $\K$ be a field containing a primitive $p^{th}$-root of unity.
 Theorem~\ref{prop:Zmxlabeliannormal}, together with \cite[Thm.~3.1]{ek98:abeliansbgps}
  and \cite[Thm.~4.6]{en94:pro2gps}, implies that the $\theta_{\K,p}$-center
 of the maximal pro-$p$ Galois group $G_{\K}(p)$ is the inertia group of the maximal {\sl $p$-henselian valuation}
 admitted by $\K$.
\end{rem}

%%%%%%%%%%%%%%%%%%%%%%%5%%%%%%%%%%%%%%%%%%%%%%%%%%%%%%%%%%%%%%%%%%%%%%%%%%%

\subsection{Isolated subgroups}
\label{ss:isolated}
Let $G$ be a pro-$p$ group, and let $S\subseteq G$ be a closed subgroup
of $G$. Then $S$ is called {\sl isolated}, if for all $g\in G$ for which there exists
$k\geq 1$ such that $g^{p^k}\in S$ follows that $g\in S$. Hence a closed normal
subgroup $N$ of $G$ is isolated if, and only if, $G/N$ is torsion free.

\begin{prop}
\label{prop:isoZt}
Let $(G,\theta)$ be an oriented Bloch-Kato pro-$p$ group. In the case $p=2$ assume further that
$\image(\theta)\subseteq 1+4\Z_2$ and that Ê$\theta$ is $1$-cyclotomic. 
Then $\Zen_\theta(G)$ is an isolated subgroup of $G$.
\end{prop}

\begin{proof}
Suppose there exists $x\in G\smallsetminus\Zen_\theta(G)$ and $k\geq1$ such that $x^{p^k}\in \Zen_\theta(G)$.
By changing the element $x$ if necessary, we may assume that $k=1$, i.e., $x^p\in\Zen_\theta(G)$.
As $G$ is torsion free (cf. Corollary~\ref{cor:torsion}), one has that $x^p\neq1$.

For an arbitrary $g\in G$, the subgroup $C(g)=\cl(\langle\,  g,x \,\rangle)\subseteq G$ is not free, as $gx^pg^{-1}=x^{p\theta(g)}$.
Thus, from Fact~\ref{fact:tits}  one concludes that $C(g)$ is $\theta\vert_{C(g)}$-abelian.
Moreover, as $\image(\theta)$ is torsion-free, $\theta(x^p)=\theta(x)^p=1$ implies that $$x\in\kernel(\theta\vert_{C(g)})=\Zen_{\theta\vert_{C(g)}}(C(g)) .$$
Thus, $x\in\bigcap_{g\in G}\Zen_{\theta_{C(g)}}(C(g))\subseteq\Zen_\theta(G)$.
\end{proof}

Proposition~\ref{prop:isoZt} generalises to profinite groups as follows.

\begin{cor}\label{cor:isoZt}
Let $(G,\theta)$ be a torsion free $p$-oriented Bloch-Kato profinite group.
For $p=2$ assume also that $\image(\theta)\subseteq 1+4\Z_2$ and that
$\theta$ is $1$-cyclotomic.
Then $\Zen_\theta(G)$ is an isolated subgroup of $G$.
\end{cor}

\begin{proof}
Let $x\in\Zen_\theta(G)$, $y\in G$ and $n\in\N$ such that $x=y^n$.
Then $Y=\cl(\langle\,  y \,\rangle)$ is pro-cyclic and virtually pro-$p$.
Thus, as $G$ is torsion free by hypothesis, $Y$ is a cyclic pro-$p$ group,
and $n$ is a $p$-power.
Let $P\in\mathrm{Syl}_p(G)$ be a pro-$p$ Sylow subgroup of $G$ containing $Y$.
Then $(P,\theta\vert_P)$ satisfies the hypothesis of Proposition~\ref{prop:isoZt},
which yields the claim.
\end{proof}

%%%%%%%%%%%%%%%%%%%%%%%%%%%%%%%%5

\subsection{Split extensions}
\label{ss:split}

%A pro-$p$ group $G$ is said to be {\sl $H^\bullet$-quadratic},
%if $H^\bullet(G,\F_p)$ endowed with the cup-product is a quadratic $\F_p$-algebra.
%One has the following property.

\begin{prop}
\label{prop:norm}
Let $(G,\theta)$ be a $p$-oriented Bloch-Kato pro-$p$ group 
of finite cohomological dimension satisfying $\image(\theta)\subseteq1+p\Z_p$
\textup{(}resp. $\image(\theta)\subseteq1+4\Z_2$ if $p=2$\textup{)},
and let $Z$ be a closed normal subgroup of $G$ isomorphic to $\Z_p$ such that
$G/Z$ is torsion free.
Then $Z\not\subseteq G^p[G,G]$.
\end{prop}

\begin{proof} Let $d=\ccd_p(G)$.
As $\ccd(Z)=1$, and as $H^1(Z,\F_p)\simeq \F_p$,
one has $\vcd_p(G/Z)=d-1$ (cf. \cite{wz:cd}). Thus, as $G/Z$ is torsion free,
J-P.~Serre's theorem (cf. \cite{ser:cd}) implies that $\ccd_p(G/Z)=d-1$.

Suppose that $Z\subseteq G^p[G,G]$. Then 
$\ifl_{G,Z}^1\colon H^1(G/Z,\F_p)\to H^1(G,\F_p)$ is an isomorphism.
For $\chi\in H^1(G,\F_p)$, set $\bar\chi\in H^1(G/Z,\F_p)$ such that $\chi=\ifl_{G,Z}^1(\bar\chi)$.
Then, by \cite[Prop.~1.5.3]{nsw:cohn} one has 
\[
 \chi_1\cup\ldots\cup\chi_k=\ifl_{G,Z}^1(\bar\chi_1)\cup\ldots\cup\ifl_{G,Z}^1(\bar\chi_k)=
 \ifl_{G,Z}^k(\bar\chi_1\cup\ldots\cup\bar\chi_k)
\]
for any $\chi_1,\ldots,\chi_k\in H^1(G,\F_p)$, i.e., 
\begin{equation}
\label{eq:norm}
\ifl_{G,Z}^k\colon H^k(G/Z,\F_p)\longrightarrow H^k(G,\F_p)
\end{equation} 
is surjective for all $k\geq 0$.
Let
\begin{equation}
\label{eq:norm2}
(E_r^{st},d_r)\Rightarrow H^{s+t}(G,\F_p),\qquad E_2^{st}=H^s\left(G/Z,H^t(Z,\F_p)\right)
\end{equation}
denote the Hochschild-Serre spectral sequence associated to
the extension of pro-$p$ groups $Z\to G\to G/Z$ with coefficients in the
discrete $G$-module $\F_p$.
We claim that $E^{st}_\infty$ is concentrated on the buttom row, i.e., $E_\infty^{st}=0$ for all $t\geq1$.
Since $\ccd_p(Z)=1$ and $\ccd_p(G/Z)=d-1$, one has $E_2^{st}=0$ for $t\geq2$ or $s\geq d$. Hence, $d_r^{st}$ is the 0-map for every $s,t\geq0$ and $r\geq3$, i.e., $E_\infty^{st}\simeq E_3^{st}$.
The total complex
$\mathbf{tot}_\bullet(E^{\bullet\bullet}_\infty)$
of the graded $\F_p$-bialgebra $E_\infty^{\bullet\bullet}$ coincides with $H^\bullet(G,\F_p)$,
which is quadratic by hypothesis.
Thus $E_\infty^{\bullet\bullet}$ is generated by 
$$\mathbf{tot}_1(E^{\bullet\bullet}_\infty)=E_\infty^{1,0}=E_2^{1,0}.$$
Hence, $E_3^{st}=0$ for $t\geq1$.

On the other hand, $ H^1(Z,\F_p)$ is a trivial $G/Z$-module isomorphic to $\F_p$, and thus,
as $\ccd_p(G/Z)=d-1$, one has
\begin{equation}
\label{eq:e2}
E^{d-1,1}_2=H^{d-1}\left(G/Z,H^1(Z,\F_p)\right)\not=0.
\end{equation}
Moreover, $d_2^{d-1,1}$ is the 0-map, thus $E_3^{d-1,1}=\ker(d_2^{d-1,1})=E_\infty^{d-1,1}\neq0$,
a contradiction, and this yields the claim.
\end{proof}

Proposition~\ref{prop:norm} has the following consequence.

\begin{prop}
\label{prop:split}
Let $(G,\theta)$ be a $p$-oriented Bloch-Kato pro-$p$ group
(resp. virtual pro-$p$ group) of finite cohomological $p$-dimension,
and let $Z$ be a closed normal subgroup of $G$ isomorphic to $\Z_p$ such that
$G/Z$ is torsion free.
Then there exists a $Z$-complement $C$ in $G$, i.e., the extension of profinite groups
\begin{equation}
\label{eq:ext}
\xymatrix{
\triv\ar[r]&Z\ar[r]&G\ar[r]&G/Z\ar[r]&\triv
}
\end{equation}
splits.
\end{prop}

\begin{proof}
Assume first that $G$ is a pro-$p$ group.
By Proposition~\ref{prop:norm}, one has that $Z\not\subseteq \Phi(G)=G^p[G,G]$.
Hence there exists a maximal closed subgroup $C_1$ of $G$ such that
$$C_1Z=G\qquad\text{and}\qquad Z_1=C_1\cap Z=Z^p.$$ Moreover, $Z_1$ is a closed
normal subgroup in $C_1$ such that $C_1/Z_1$ is torsion free and $Z_1\simeq \Z_p$.
From Proposition~\ref{prop:norm} again, one concludes that $Z_1\not\subseteq\Phi(C_1)$.
Thus repeating this process one finds open subgroup $C_k$ of $G$ of index $p^k$
such that $$C_k\,Z=G\qquad\text{and}\qquad Z_k=C_k\cap Z=Z^{p^k}.$$
Hence $C=\bigcap_{k\geq 1} C_k$ is a $Z$-complement in $G$.

If $G$ is a $p$-oriented virtual pro-$p$ group, then $G$ is a
$\bSigma$-virtual pro-$p$ group for $\bSigma=\image(\hat{\theta})$
(cf. \ref{ss:Obsig}), and thus corresponds to $(O_p(G),\theta^\circ,\gamma)$
in alternative form. In particular, the maximal subgroup $C_1$ and hence all closed subgroups
$C_k$ can be chosen to be $\bSigma$-invariant (cf. Proposition~\ref{prop:maxSig}).
Hence $C=\bigcap_{k\in\N} C_k$ carries canonically a left $\bSigma$-action, and thus defines a $Z$ complement $H=C\rtimes\bSigma$ in $G$.
\end{proof}

The proof of Theorem~\ref{thmC} can be deduced from Proposition~\ref{prop:split} as follows.

\begin{proof}[Proof of Theorem \ref{thmC}]
Assume first that $G$ is either pro-$p$, or virtually pro-$p$.
To prove statement (i) (and (ii)), we proceed by induction on $d=\ccd_p(G)=\ccd(G)$. For $d=1$, $G$ is free (resp. virtually free) (cf. \cite[Prop.~3.5.17]{nsw:cohn}), and thus $\Zen_\theta(G)=\{1\}$.
So assume that $d\geq 1$,
and that the claim holds for $d-1$.
Note that $\Zen_\theta(G)$ is a finitely generated abelian pro-$p$
group satisfying $$d_\circ=d(\Zen_\theta(G))=\ccd_p(\Zen_\theta(G))\leq d.$$ If $d_\circ=0$, there is nothing to prove.
If $d_\circ\geq 1$, $\Zen_\theta(G)$ contains an isolated closed subgroup $Z$
satisfying $d(Z)=1$. By definition, $Z$ is normal in $G$. Hence
Proposition~\ref{prop:split} implies that there exists a subgroup $C\subseteq G$
satisfying $C\cap Z=\triv$ and $C\, Z=G$.
As $C\simeq G/Z$, the main result of \cite{wz:cd} implies that $\ccd(C)=\vcd(C)=d-1$.
Since $\Zen_{\theta\vert_C}(C)\,Z=\Zen_\theta(G)$, the claim then follows by induction.

To prove statement (iii), let $G^\circ=\kernel(\hat{\theta}\colon G\to \F_p^\times)$ and
$\baG^\circ=\kernel(\hat{\btheta}\colon \baG\to \F_p^\times)$, and put 
$\baO=O^p(\baG^\circ)$ and 
\begin{equation}
\label{eq:X0}
O=\{\,g\in G^\circ\mid g\Zen_\theta(G)\in\baO^p(\baG)\,\}.
\end{equation}
Then,  by construction, $\image(\hat{\btheta}\vert_{\baO})$ is a pro-$p$ group
and hence trivial. In particular, the left $\F_p\dbl\baO\dbr$-module $\F_p(1)$
is the trivial module. Thus, as $\baO$ is $p$-perfect, one concludes
that 
\begin{equation}
\label{eq:X1}
H^1(\baO,\F_p(1))=0.
\end{equation}
By hypothesis, $(\baG,\btheta)$ is Bloch-Kato, and therefore $(\baO,\bone)$ is Bloch-Kato.
Hence \eqref{eq:X1} yields that 
\begin{equation}
\label{eq:X2}
H^k(\baO,\F_p(j))=H^k(\baO,\F_p(0))=0
\end{equation}
for all positive integers $k$, $j$. Note that $\Z_p(1)$ is the trivial $\Z_p\dbl\baO\dbr$-module isomorphic to $\Z_p$ as abelian pro-$p$ group. The cyclotomicity of $(\baO,\bone)$ implies
that $H^2(\baO,\Z_p(1))$ is $p$-torsion free, and from the exact sequence
\begin{equation}
\label{eq:X3}
\xymatrix{0\ar[r]&H^2(\baO,\Z_p(1))\ar[r]^{\cdot p}&H^2(\baO,\Z_p(1))\ar[r]&
H^2(\baO,\F_p(1))\ar[r]&0}
\end{equation}
one concludes that 
\begin{equation}
\label{eq:X4}
H^2(\baO,\Z_p(1))=0.
\end{equation}
By hypothesis, $\ccd_p(\Zen_\theta(G))\leq\ccd_p(G)<\infty$, and thus
$\Zen_\theta(G)\simeq\Z_p(1)^r$ is a trivial left $\Z_p\dbl\baO\dbr$-module
and a finitely generated free (abelian pro-$p$ group).
Hence
\begin{equation}
\label{eq:X5}
H^2(\baO,\Zen_\theta(G))=0,
\end{equation}
which implies that
\begin{equation}
\label{eq:X6}
\xymatrix{
\triv\ar[r]&\Zen_\theta(G)\ar[r]&O\ar[r]^{\pi}&\baO\ar[r]&\triv}
\end{equation}
is a split short exact sequence of profinite groups.
From this fact one concludes that
\begin{equation}
\label{eq:X7}
O=\Zen_\theta(G)\cdot O^p(G^\circ)\qquad\text{and}
\qquad \Zen_\theta(G)\cap O^p(G^\circ)=\triv.
\end{equation}
Let $\tG=G/O^p(G^\circ)$. Then for all abelian pro-$p$ groups $M$
with a continuous left $\Z_p\dbl\tG\dbr$-action inflation induces an isomorphism
in cohomology
\begin{equation}
\label{eq:X8}
\mathrm{inf}_{\tG}^G(-)\colon H_{\mathrm{cts}}^k(\tG,M)\longrightarrow H_{\mathrm{cts}}^k(G,M)
\end{equation}
(cf. Proposition~\ref{propB}).
Moreover, as $\theta\vert_O=\bone$ is the constant $1$ function, $\theta$ induces
a $p$-orientation $\ttheta\colon\tG\to\Z_p^\times$ on $\tG$.
In particular, from \eqref{eq:X8} one concludes that $\ccd_p(\tG)<\infty$, and that $(\tG,\ttheta)$ is cyclotomic and Bloch-Kato.
Thus,  by part (i), the exact sequence of virtual pro-$p$ groups
\begin{equation}
\label{eq:X9}
\xymatrix{
\triv\ar[r]& \dfrac{\Zen_\theta(G)O^p(G^\circ)}{O^p(G^\circ)}\ar[r]&\tG\ar[r]^{\tpi}&\baG/\baO\ar[r]&\triv}
\end{equation}
splits. Let $\tH\subset\tG$ be a complement for $\Zen_\theta(G)O^p(G^\circ)/O^p(G^\circ)$
in $\tG$, and let
\begin{equation}
\label{eq:X10}
H=\{\,g\in G^\circ\mid gO^p(G^\circ)\in\tH\,\}.
\end{equation}
Then, by construction, $H\cap\Zen_\theta(G)O^p(G^\circ)\subseteq O^p(G^\circ)$.
Thus $HO^p(G^\circ)$ is a complement of $\Zen_\theta(G)$ in $G$.
\end{proof}

Finally, we ask whether the converse of Theorem~\ref{thm:fibco} holds true.

\begin{ques}
 Let $(G,\theta)$ be a cyclotomically $p$-oriented Bloch-Kato pro-$p$ group, and suppose that 
 \[H^\bullet(G,\F_p)\simeq H^\bullet(C,\F_p)\otimes\Lambda_\bullet(V),\]
 for some subgroup $C\subseteq G$ and some nontrivial subspace $V\subseteq H^1(G,\F_p)$. Does there exist an isolated closed subgroup $\Zen\subseteq\Zen_\theta(G)$ such that
 $G=C\Zen$ and $\Zen/\Zen^p\simeq V^\ast=\Hom(V,\F_p)$?
\end{ques}

%%%%%%%%%%%%%%%%%%%%%%%%%%%%%%%%%%%%%%%%%%%%%%%%%%%%%%%%%%%%%%%%%%%%5

\subsection{The elementary type conjecture}
\label{ss:ETC}

In order to formulate a conjecture concerning the maximal pro-$p$ Galois groups of fields, I.~Efrat introduced in \cite{ido:small}
the class $\mathcal{C}_{\mathrm{FG}}$ of $p$-oriented pro-$p$ groups (resp. cyclotomic pro-$p$ pairs) of {\sl elementary type}.

This class consists of all finitely generated $p$-oriented pro-$p$ groups which can be constructed from $\Z_p$ and Demu\v skin groups using coproducts and fibre products (cf. \cite[\S~3]{ido:small}).

Efrat's {\sl elementary type conjecture} asks whether every pair $(G_{\K}(p),\theta_{\K,p})$ 
for which $\K$ contains a primitive $p^{th}$-root of unity and $G_{\K}(p)$ is finitely generated, belongs to $\mathcal{C}_{\mathrm{FG}}$
(see \cite{ido:ETC}, and also \cite{JW:ETC} for the case $p=2$).
This conjecture originates from the theory of quadratic forms 
(cf. \cite{mar:ETC}, \cite[p.~268]{ido:miln}).%

One may extend slightly Efrat's class by defining the class $\mathcal{E}_{\mathrm{CO}}$ of {\sl cyclotomically $p$-oriented Bloch-Kato
pro-$p$ groups of elementary type}
to be the smallest class of cyclotomically $p$-oriented pro-$p$ groups containing
\begin{itemize}
 \item[(a)] $(F,\theta)$, with $F$ a finitely generated free pro-$p$ group
and $\theta\colon F\to\Z_p^\times$ any $p$-orientation;
 \item[(b)] $(G,\eth_G)$, with $G$ a Demu\v{s}kin pro-$p$ group;
 \item[(c)] $(\Z/2\Z,\theta)$, with $\image(\theta)=\{\pm1\}$ in case that $p=2$;
 \end{itemize}
and which is closed under coproducts and under fibre products with respect to finitely generated split $\theta$-abelian pro-$p$ groups, i.e.,  if  $(G_1,\theta_1)$ and $(G_2,\theta_2)$ are contained in $\mathcal{E}_{\mathrm{CO}}$, then
\begin{itemize}
 \item[(d)] $(G,\theta)=(G_1,\theta_1)\amalg(G_2,\theta_2)\in\mathcal{E}_{\mathrm{CO}}$; and
 \item[(e)] $(G,\theta)=\Z_p\rtimes_{\theta_1} (G_1,\theta_1) \in\mathcal{E}_{\mathrm{CO}}$.
\end{itemize}

Question~\ref{ques:ETC} asks whether every finitely generated cyclotomically $p$-oriented Bloch-Kato pro-$p$ group belongs to the class $\mathcal{E}_{\mathrm{CO}}$.
By Theorem~\ref{thmA}, Question~\ref{ques:ETC} is stronger than Efrat's elementary type conjecture. Nevertheless, it is stated in purely group theoretic terms. 

\begin{rem}
Recently, Question~\ref{ques:ETC} has received a positive solution in the class of {\sl trivially $p$-oriented right-angled Artin pro-$p$ groups}: I.~Snopce and P.A.~Zalesski{\u \i} proved that the only indecomposable right-angled Artin pro-$p$ group which is Bloch-Kato and cyclotomically $p$-oriented is $(\Z_p,\bone)$ (cf. \cite{SZ:RAAGs}).
\end{rem}

%%%%%%%%%%%%%%%%%%%%%%%%%%%%%%%%%%%%%%%%%%%%%%%%%%%%%%%%%%%%%%%%%%%%%%%5

%%%%%%%%%%%%%%%%%%%5

%%%%%%%%%%%%%%%%%%%%%%%%%%%%%%%%%%%%%%5
%%%%%%%%5
%%%%%%%%%%%%%%%%%%%%%%%%%%%%%%%%%%%%

% ------------------------------------------------------------------------

\subsection*{Acknowledgments}

{\small The authors are grateful to
I.~Efrat, for the interesting discussion they had together at the Ben-Gurion University of the Negev in 2016; and to D.~Neftin and I.~Snopce, for their interest.
Also, the first-named author wishes to thank M.~Florence and P.~Guillot for the discussions on the preprint \cite{smooth}.}

\bibliographystyle{plain}
\bibliography{BK}{}

\def\cprime{$'$}
\begin{thebibliography}{10}

\bibitem{bk:ihes}
S.~Bloch and K.~Kato.
\newblock {$p$}-adic {\'e}tale cohomology.
\newblock {\em Inst. Hautes {\'E}tudes Sci. Publ. Math.}, 63:107--152, 1986.

\bibitem{coh:brown}
K.~S. Brown.
\newblock {\em Cohomology of groups}, volume~87 of {\em Graduate Texts in
  Mathematics}.
\newblock Springer-Verlag, New York-Berlin, 1982.

\bibitem{brum:pseudo}
A.~Brumer.
\newblock Pseudocompact algebras, profinite groups and class formations.
\newblock {\em J. Algebra}, 4:442--470, 1966.

\bibitem{formrealfields}
T.~C. Craven and T.~L. Smith.
\newblock Formally real fields from a {G}alois-theoretic perspective.
\newblock {\em J. Pure Appl. Algebra}, 145(1):19--36, 2000.

\bibitem{smooth}
C.~De~Clercq and M.~Florence.
\newblock Lifting theorems and smooth profinite groups.
\newblock preprint, available at {\tt arxiv:1710.10631}, 2017.

\bibitem{ddms:padic}
J.~D. Dixon, M.~P.~F. du~Sautoy, A.~Mann, and D.~Segal.
\newblock {\em Analytic pro-{$p$} groups}, volume~61 of {\em Cambridge Studies
  in Advanced Mathematics}.
\newblock Cambridge University Press, Cambridge, second edition, 1999.

\bibitem{ido:ETC}
I.~Efrat.
\newblock Orderings, valuations, and free products of {G}alois groups.
\newblock {\em Sem. Structure Alg{\'e}briques Ordonn{\'e}es, {U}niv. {P}aris
  {VII}}, 1995.

\bibitem{efrat:ETC}
I.~Efrat.
\newblock Pro-{$p$} {G}alois groups of algebraic extensions of {$\bold Q$}.
\newblock {\em J. Number Theory}, 64(1):84--99, 1997.

\bibitem{ido:small}
I.~Efrat.
\newblock Small maximal pro-{$p$} {G}alois groups.
\newblock {\em Manuscripta Math.}, 95(2):237--249, 1998.

\bibitem{ido:miln}
I.~Efrat.
\newblock {\em Valuations, orderings, and {M}ilnor {$K$}-theory}, volume 124 of
  {\em Mathematical Surveys and Monographs}.
\newblock American Mathematical Society, Providence, RI, 2006.

\bibitem{eq:kummer}
I.~Efrat and C.~Quadrelli.
\newblock The {K}ummerian property and maximal pro-{$p$} {G}alois groups.
\newblock {\em J. Algebra}, 525:284--310, 2019.

\bibitem{ek98:abeliansbgps}
A.~J. Engler and J.~Koenigsmann.
\newblock Abelian subgroups of pro-{$p$} {G}alois groups.
\newblock {\em Trans. Amer. Math. Soc.}, 350(6):2473--2485, 1998.

\bibitem{en94:pro2gps}
A.~J. Engler and J.~B. Nogueira.
\newblock Maximal abelian normal subgroups of {G}alois pro-{$2$}-groups.
\newblock {\em J. Algebra}, 166(3):481--505, 1994.

\bibitem{friedjarden:FA}
M.~Fried and M.~Jarden.
\newblock {\em Field arithmetic}, volume~11 of {\em Ergebnisse der Mathematik
  und ihrer Grenzgebiete. 3. Folge. A Series of Modern Surveys in Mathematics
  [Results in Mathematics and Related Areas. 3rd Series. A Series of Modern
  Surveys in Mathematics]}.
\newblock Springer-Verlag, Berlin, third edition, 2008.
\newblock Revised by Jarden.

\bibitem{JW:ETC}
B.~Jacob and R.~Ware.
\newblock A recursive description of the maximal pro-{$2$} {G}alois group via
  {W}itt rings.
\newblock {\em Math. Z.}, 200(3):379--396, 1989.

\bibitem{labute:demushkin}
J.~P. Labute.
\newblock Classification of {D}emushkin groups.
\newblock {\em Canad. J. Math.}, 19:106--132, 1967.

\bibitem{lazard:padic}
M.~Lazard.
\newblock Groupes analytiques {$p$}-adiques.
\newblock {\em Inst. Hautes \'{E}tudes Sci. Publ. Math.}, 26:389--603, 1965.

\bibitem{mcl:hom}
S.~Mac~Lane.
\newblock {\em Homology}.
\newblock Classics in Mathematics. Springer-Verlag, Berlin, 1995.
\newblock Reprint of the 1975 edition.

\bibitem{addistruct}
L.~Mah\'e, J.~Min\'a\v{c}, and T.~L. Smith.
\newblock Additive structure of multiplicative subgroups of fields and {G}alois
  theory.
\newblock {\em Doc. Math.}, 9:301--355, 2004.

\bibitem{mar:ETC}
M.~Marshall.
\newblock The elementary type conjecture in quadratic form theory.
\newblock In {\em Algebraic and arithmetic theory of quadratic forms}, volume
  344 of {\em Contemp. Math.}, pages 275--293. Amer. Math. Soc., Providence,
  RI, 2004.

\bibitem{JanTan:massey2}
J.~Min\'a\v{c} and N.~D. T\^an.
\newblock Triple {M}assey products vanish over all fields.
\newblock {\em J. Lond. Math. Soc. (2)}, 94(3):909--932, 2016.

\bibitem{JanTan:massey1}
J.~Min\'a\v{c} and N.~D. T\^an.
\newblock Triple {M}assey products and {G}alois theory.
\newblock {\em J. Eur. Math. Soc. (JEMS)}, 19(1):255--284, 2017.

\bibitem{nsw:cohn}
J.~Neukirch, A.~Schmidt, and K.~Wingberg.
\newblock {\em Cohomology of number fields}, volume 323 of {\em Grundlehren der
  Mathematischen Wissenschaften [Fundamental Principles of Mathematical
  Sciences]}.
\newblock Springer-Verlag, Berlin, second edition, 2008.

\bibitem{pp:quad}
A.~Polishchuk and L.~Positselski.
\newblock {\em Quadratic algebras}, volume~37 of {\em University Lecture
  Series}.
\newblock American Mathematical Society, Providence, RI, 2005.

\bibitem{cq:BK}
C.~Quadrelli.
\newblock Bloch-{K}ato pro-{$p$} groups and locally powerful groups.
\newblock {\em Forum Math.}, 26(3):793--814, 2014.

\bibitem{cq:phd}
C.~Quadrelli.
\newblock {\em Cohomology of Absolute Galois Groups}.
\newblock PhD thesis, University of Western Ontario, 2015.

\bibitem{ribzal:book}
L.~Ribes and P.~A. Zalesski{\u \i}.
\newblock {\em Profinite groups}, volume~40.
\newblock Springer-Verlag, Berlin, 2000.

\bibitem{roos}
J-E. Roos.
\newblock Sur les foncteurs d\'{e}riv\'{e}s de {$\underleftarrow\lim$}.
  {A}pplications.
\newblock {\em C. R. Acad. Sci. Paris}, 252:3702--3704, 1961.

\bibitem{rost:BK}
M.~Rost.
\newblock Norm varieties and algebraic cobordism.
\newblock In {\em Proceedings of the {I}nternational {C}ongress of
  {M}athematicians. {V}ol. {II} ({B}eijing, 2002)}, pages 77--85. Higher Ed.
  Press, Beijing, 2002.

\bibitem{ser:cd}
J-P. Serre.
\newblock Sur la dimension cohomologique des groupes profinis.
\newblock {\em Topology}, 3:413--420, 1965.

\bibitem{ser:gal}
J-P. Serre.
\newblock {\em Galois cohomology}.
\newblock Springer-Verlag, Berlin, 1997.

\bibitem{SZ:RAAGs}
I.~Snopce and P.~A. Zalesski{\u \i}.
\newblock Right-angled artin pro-$p$ groups.
\newblock preprint, available at {\tt arXiv:2005.01685}, 2020.

\bibitem{sw:cpg}
P.~Symonds and Th. Weigel.
\newblock Cohomology of {$p$}-adic analytic groups.
\newblock In {\em New horizons in pro-{$p$} groups}, volume 184 of {\em Progr.
  Math.}, pages 349--410. Birkh{\"a}user Boston, Boston, MA, 2000.

\bibitem{tate:miln}
J.~Tate.
\newblock Relations between {$K_{2}$} and {G}alois cohomology.
\newblock {\em Invent. Math.}, 36:257--274, 1976.

\bibitem{tits:alt}
J.~Tits.
\newblock Free subgroups in linear groups.
\newblock {\em J. Algebra}, 20:250--270, 1972.

\bibitem{voev}
V.~Voevodsky.
\newblock On motivic cohomology with {$\bold Z/l$}-coefficients.
\newblock {\em Ann. of Math. (2)}, 174(1):401--438, 2011.

\bibitem{wadsworth}
A.~R. Wadsworth.
\newblock {$p$}-{H}enselian field: {$K$}-theory, {G}alois cohomology, and
  graded {W}itt rings.
\newblock {\em Pacific J. Math.}, 105(2):473--496, 1983.

\bibitem{ware}
R.~Ware.
\newblock Galois groups of maximal {$p$}-extensions.
\newblock {\em Trans. Amer. Math. Soc.}, 333(2):721--728, 1992.

\bibitem{weibel:book}
Ch.~A. Weibel.
\newblock {\em An introduction to homological algebra}, volume~38 of {\em
  Cambridge Studies in Advanced Mathematics}.
\newblock Cambridge University Press, Cambridge, 1994.

\bibitem{weibel}
Ch.~A. Weibel.
\newblock The norm residue isomorphism theorem.
\newblock {\em J. Topol.}, 2(2):346--372, 2009.

\bibitem{wei:stand}
Th. Weigel.
\newblock Frattini extensions and class field theory.
\newblock In {\em Groups {S}t. {A}ndrews 2005. {V}ol. 2}, volume 340 of {\em
  London Math. Soc. Lecture Note Ser.}, pages 661--684. Cambridge Univ. Press,
  Cambridge, 2007.

\bibitem{wei:2deg}
Th. Weigel.
\newblock A characterization of powerful pro-2 groups by their cohomology.
\newblock preprint, 2016.

\bibitem{wz:cd}
Th. Weigel and P.~A. Zalesski{\u \i}.
\newblock {Profinite groups of finite cohomological dimension.}
\newblock {\em C. R., Math., Acad. Sci. Paris}, 338(5):353--358, 2004.

\end{thebibliography}

%%--------------------Here the manuscript ends--------------------------------

% ------------------------------------------------------------------------
\end{document}
% ------------------------------------------------------------------------